\documentclass{amsart}

\usepackage{amsmath, amssymb, amsthm}
\usepackage{bbm}

\newtheorem{theorem}{Theorem}[section]
\newtheorem{lemma}[theorem]{Lemma}
\newtheorem{proposition}[theorem]{Proposition}
\newtheorem{corollary}[theorem]{Corollary}
\newtheorem{claim}[theorem]{Claim}

\newtheorem{fact}[theorem]{Fact}
\newtheorem{question}[theorem]{Question}

\theoremstyle{definition}
\newtheorem{definition}[theorem]{Definition}
\newtheorem{remark}[theorem]{Remark}

\newcommand{\cf}{\mathrm{cf}}
\newcommand{\dom}{\mathrm{dom}}
\newcommand{\bb}{\mathbb}

\newcommand{\otp}{\mathrm{otp}}

\newcommand{\nacc}{\mathrm{nacc}}
\newcommand{\acc}{\mathrm{acc}}
\newcommand{\pred}{\mathrm{pred}}
\newcommand{\height}{\mathrm{ht}}

\begin{document}
\title{Aronszajn trees, square principles, and stationary reflection}
\author{Chris Lambie-Hanson}
\address{Einstein Institute of Mathematics, Hebrew University of Jerusalem \\ 
Jerusalem, 91904, Israel}
\email{clambiehanson@math.huji.ac.il}
\thanks{This research was completed while the author was a Lady Davis Postdoctoral 
  Fellow. The author would like to thank the Lady Davis Fellowship Trust and the 
  Hebrew University of Jerusalem. The author would also like to thank Assaf Rinot 
  for some helpful conversations.}
\begin{abstract}
  We investigate questions involving Aronszajn trees, square principles, and stationary 
  reflection. We first consider two strengthenings of $\square(\kappa)$ 
  introduced by Brodsky and Rinot for the purpose of constructing 
  $\kappa$-Souslin trees. Answering a question of Rinot, we prove that 
  the weaker of these strengthenings is compatible with stationary 
  reflection at $\kappa$ but the stronger is not. We then prove that, 
  if $\mu$ is a singular cardinal, $\square_\mu$ implies the existence of a 
  special $\mu^+$-tree with a $\cf(\mu)$-ascent path, 
  thus answering a question of L\"{u}cke.
\end{abstract}
\maketitle

\section{Introduction}

In this paper, we address recent questions of Rinot and L\"{u}cke involving 
trees and square sequences. We begin by reviewing the cast of characters.

A partial order $(T, <_T)$ is a \emph{tree} if, for 
all $t \in T$, the set $\pred_T(t) := \{s \in T \mid s<_T t\}$ is 
well-ordered by $<_T$. We often abuse notation and refer to such a tree 
as $T$ rather than $(T, <_T)$. If $T$ is a tree and $t \in T$, then 
$\height_T(t) = \otp(\pred_T(t), <_T)$. For an ordinal $\alpha$, $T_\alpha$ 
is the set of $t \in T$ such that $\height_T(t) = \alpha$, and the \emph{height} 
of the tree, $\height(T)$, is the least $\alpha$ such that $T_\alpha = \emptyset$. 
$T_{<\alpha}, T_{\leq \alpha}$, etc. are defined in the obvious way. A subset $b$ 
of a tree $T$ is a \emph{chain} in $T$ if $b$ is linearly ordered by $<_T$. If a chain 
$b$ in $T$ is downward closed under $<_T$, it is called a \emph{branch} through $T$. A 
branch $b$ through a tree $T$ is a \emph{cofinal branch} if, for all 
$\alpha < \height(T)$, $b \cap T_\alpha \not= \emptyset$. A subset of $T$ is an \emph{antichain} 
if its elements are pairwise $<_T$-incomparable. If $\kappa$ 
is an infinite, regular cardinal and $T$ is a tree, then $T$ is a \emph{$\kappa$-tree} 
if $\height(T) = \kappa$ and, for all $\alpha < \kappa$, $|T_\alpha| < \kappa$. 
If $T$ is a $\kappa$-tree and $T$ has no cofinal branch, then $T$ is said to be a 
\emph{$\kappa$-Aronszajn} tree.

We will be interested in two particular types of $\kappa$-Aronszajn trees: $\kappa$-Souslin trees and special 
$\kappa$-trees.

\begin{definition}
  Let $\kappa$ be an uncountable, regular cardinal. A \emph{$\kappa$-Souslin tree} is a 
  $\kappa$-Aronszajn tree with no antichains of size $\kappa$.
\end{definition}

The following definition is due to Todorcevic. In what follows, if $T$ is a tree of height 
$\kappa$ and $S \subseteq \kappa$, then $T \restriction S$ is 
$\bigcup\{T_\alpha \mid \alpha \in S\}$, equipped with the restriction of $<_T$.

\begin{definition} [Todorcevic, \cite{Todorcevic1981}]
  Suppose $\kappa$ is an uncountable, regular cardinal and $T$ is a tree of height $\kappa$.
  \begin{enumerate}
    \item If $S \subseteq \kappa$ and $r:T \restriction S \rightarrow T$, then $r$ is 
      \emph{regressive} if, for every non-minimal $t \in T \restriction S$, 
      $r(t) <_T t$.
    \item If $S \subseteq \kappa$, $S$ is \emph{non-stationary with respect to $T$} if 
      there is a regressive $r:T \restriction S \rightarrow T$ such that, for every 
      $t \in T$, there is $\mu_t < \kappa$ and a function $c_t:r^{-1}(t) \rightarrow \mu_t$ 
      such that $c_t$ is injective on chains in $T$.
    \item $T$ is \emph{special} if $\kappa$ is non-stationary with respect to $T$.
  \end{enumerate}
\end{definition}

If $\kappa = \mu^+$ for some cardinal $\mu$, then it can be shown that this coincides with 
the classical definition stating that a tree $T$ of height $\kappa$ is special if there is a 
function $f:T \rightarrow \mu$ that is injective on chains in $T$. It is easily seen that 
a special tree cannot have a cofinal branch and also cannot be Souslin.

We will be using a variety of square principles. The earliest such principle was introduced 
by Jensen \cite{jensen}; the generalization given here is due to Schimmerling.

\begin{definition} [Schimmerling, \cite{schimmerling}]
  Suppose $\mu$ and $\lambda$ are cardinals. A sequence $\vec{\mathcal{C}} = 
  \langle \mathcal{C}_\alpha \mid \alpha < \mu^+ \rangle$ is a $\square_{\mu, < \lambda}$-sequence 
  if:
  \begin{enumerate}
    \item for all $\alpha < \mu^+$, $\mathcal{C}_\alpha$ is a set of clubs in $\alpha$ with 
      $0 < |\mathcal{C}_\alpha| < \lambda$;
    \item for all $\alpha < \beta < \mu^+$ and all $C \in \mathcal{C}_\beta$, if 
      $\alpha \in \acc(C)$, then $C \cap \alpha \in \mathcal{C}_\alpha$;
    \item for all $\alpha < \mu^+$ and all $C \in \mathcal{C}_\alpha$, $\otp(C) \leq \mu$.
  \end{enumerate}
  $\square_{\mu, < \lambda}$ is the assertion that there is a $\square_{\mu, < \lambda}$-sequence. 
  $\square_{\mu, < \lambda^+}$ is typically denoted $\square_{\mu, \lambda}$, $\square_{\mu, 1}$ is 
  typically denoted $\square_\mu$ and is Jensen's original square principle, and $\square_{\mu, \mu}$, 
  also investigated by Jensen,
  is often denoted $\square^*_\mu$ and is known as \emph{weak square}.
\end{definition}

An immediate consequence of condition (3) in the definition of $\square_{\mu, < \lambda}$ is that, 
if $\vec{\mathcal{C}}$ is a $\square_{\mu, < \lambda}$-sequence, then $\vec{\mathcal{C}}$ does not 
have a \emph{thread}, i.e. a club $D \subseteq \mu^+$ such that, for every $\alpha \in \acc(D)$, 
$D \cap \alpha \in \mathcal{C}_\alpha$. A weakening of $\square_\mu$, due to Todorcevic, replaces 
this order-type restriction with its anti-thread consequence.

\begin{definition} [Todorcevic]
  Suppose $\kappa$ is a regular, uncountable cardinal and $\lambda > 1$ is a cardinal. 
  A sequence $\vec{\mathcal{C}} = \langle 
  \mathcal{C}_\alpha \mid \alpha < \kappa \rangle$ is a $\square(\kappa, < \lambda)$-sequence if:
  \begin{enumerate}
    \item for all $\alpha < \kappa$, $\mathcal{C}_\alpha$ is a set of clubs in $\alpha$ with 
      $0 < |\mathcal{C}_\alpha| < \lambda$;
    \item for all $\alpha < \beta < \kappa$ and all $C \in \mathcal{C}_\beta$, if $\alpha \in \acc(C)$, then 
      $C \cap \alpha \in \mathcal{C}_\alpha$;
    \item there is no club $D \subseteq \kappa$ such that, for all $\alpha \in \acc(D)$, 
      $D \cap \alpha \in \mathcal{C}_\alpha$.
  \end{enumerate}
  $\square(\kappa, < \lambda)$ is the assertion that there is a $\square(\kappa, < \lambda)$-sequence. 
  $\square(\kappa, < \lambda^+)$ is typically denoted $\square(\kappa, \lambda)$, and $\square(\kappa, 1)$ 
  is typically denoted $\square(\kappa)$.
\end{definition}

There are many connections between square principles and higher Aronszajn trees. 
Todorcevic (see \cite{todorcevic}) proved that, if $\kappa$ is a regular, uncountable 
cardinal, then $\square(\kappa, < \kappa)$ implies the existence of a $\kappa$-Aronszajn tree. 
Earlier results of Jensen, Solovay, Gregory, and Shelah combine to show that, if 
$\mu$ is uncountable, then $\mathrm{GCH} + \square_\mu$ implies the existence of a 
$\mu^+$-Souslin tree. In addition, $\square^*_\mu$ is equivalent to the existence 
of a special $\mu^+$-tree.

The last player in our story is stationary reflection. Recall that, if $\kappa$ 
is a regular, uncountable cardinal, $S \subseteq \kappa$ is stationary, and 
$\alpha < \kappa$ is an ordinal of uncountable cofinality, then $S$ \emph{reflects at $\alpha$} 
if $S \cap \alpha$ is stationary in $\alpha$. $S$ \emph{reflects} if there is $\alpha < \kappa$ 
such that $S$ reflects at $\alpha$. $\mathrm{Refl}(S)$ is the assertion that, whenever 
$T \subseteq S$ is stationary, then $T$ reflects. If $\mathcal{S}$ is a collection of stationary 
subsets of $\kappa$ and $\alpha < \kappa$ has uncountable cofinality, we say $\mathcal{S}$ 
\emph{reflects simultaneously at $\alpha$} if, for every $T \in \mathcal{S}$, $T$ reflects 
at $\alpha$. We say $\mathcal{S}$ \emph{reflects simultaneously} if there is $\alpha < \kappa$ 
such that $\mathcal{S}$ reflects simultaneously at $\alpha$. If $\lambda$ is a cardinal and 
$S \subseteq \kappa$ is stationary, then $\mathrm{Refl}(<\lambda, S)$ is the assertion that, 
whenever $\mathcal{S}$ is a collection of stationary subsets of $S$ and $|\mathcal{S}| < \lambda$, 
$\mathcal{S}$ reflects simultaneously. As usual, $\mathrm{Refl}(<\lambda^+, S)$ will be denoted 
by $\mathrm{Refl}(\lambda, S)$.

Aronszajn trees and square sequences can be seen as instances of incompactness: $\kappa$-Aronszajn 
trees have branches of every length less than $\kappa$ but no branches of length $\kappa$, and 
square sequences of length $\kappa$ cannot be extended to have length $\kappa + 1$. 
Stationary reflection and strengthenings thereof, on the other hand, are manifestly compactness principles 
and are therefore at certain odds with the existence of Aronszajn trees and square sequences. For example, 
a folklore result states that $\square_\mu$ implies the failure of $\mathrm{Refl}(S)$ for every 
stationary $S \subseteq \mu^+$. Much work has been done investigating the extent to which certain 
compactness and incompactness principles can or cannot hold simultaneously; we continue this line of research 
here.

Our notation is, for the most part, standard. We use \cite{jech} as a reference for all undefined notions. 
If $A$ is a set of ordinals, we use $\acc(A)$ to refer to $\{\beta \in A \setminus \{0\} \mid \sup(A \cap \beta) = \beta\}$ 
and let $\nacc(A) = A \setminus \acc(A)$. If $\lambda < \kappa$ are cardinals and $\lambda$ is regular, 
then $S^\kappa_\lambda = \{\beta < \kappa \mid \cf(\beta) = \lambda\}$. $S^\kappa_{<\lambda}$, $S^\kappa_{\leq \lambda}$, etc. 
are defined in the obvious way.

\subsection{Souslin tree constructions}

There have been a vast array of constructions of $\kappa$-Souslin trees that have differed based 
on the identity of $\kappa$ and any additional properties desired of the constructed tree. In 
recent work (see \cite{brodsky_rinot}), Brodsky and Rinot unify these constructions under a single framework. 
In the process, they isolate certain strengthenings of $\square(\kappa, < \lambda)$ that incorporate guessing 
properties.

\begin{definition}[Brodsky-Rinot, \cite{brodsky_rinot}]
  Suppose $\kappa$ is a regular, uncountable cardinal, $\lambda, \theta > 1$ are cardinals, and $\mathcal{S}$ 
  is a non-empty collection of stationary subsets of $\kappa$. $\langle \mathcal{C}_\alpha \mid \alpha < \kappa \rangle$ 
  is a $\boxtimes^-_\theta(\mathcal{S}, < \lambda)$-sequence if:
  \begin{enumerate}
    \item for all $\alpha < \kappa$, $\mathcal{C}_\alpha$ is a set of clubs in $\alpha$ with $0 < |\mathcal{C}_\alpha| < \lambda$;
    \item for all $\alpha < \kappa$ and all $C \in \mathcal{C}_\alpha$, $\otp(C) \leq \theta$;
    \item for all $\alpha < \beta < \kappa$ and all $C \in \mathcal{C}_\beta$, if $\alpha \in \acc(C)$, then 
      $C \cap \alpha \in \mathcal{C}_\alpha$;
    \item for every cofinal $A \subseteq \kappa$ and every $S \in \mathcal{S}$, there is $\beta \in S$ 
      such that, for all $C \in \mathcal{C}_\beta$, $\sup(\nacc(C) \cap A) = \beta$.
  \end{enumerate}
  $\boxtimes^-_\theta(\mathcal{S}, < \lambda)$ holds if there is a $\boxtimes^-_\theta(\mathcal{S}, < \lambda)$-sequence. 
  If $\theta = \kappa$, then $\theta$ is omitted from the notation. If $S \subseteq \kappa$ is stationary, we write 
  $\boxtimes^-_\theta(S, < \lambda)$ instead of $\boxtimes^-_\theta(\{S\}, < \lambda)$. $\boxtimes^-_\theta(\mathcal{S}, \lambda)$ 
  and $\boxtimes^-_\theta(\mathcal{S})$ are defined in analogy with $\square(\kappa, \lambda)$ and $\square(\kappa)$.
\end{definition}

The following is proven in \cite{brodsky_rinot}.

\begin{proposition}
  Suppose $\kappa$ is a regular, uncountable cardinal, $\lambda > 1$ is a cardinal, and $\vec{C} = \langle \mathcal{C}_\alpha \mid \alpha < \kappa 
  \rangle$ is a $\boxtimes^-(\kappa, < \lambda)$-sequence. Then $\vec{C}$ is a $\square(\kappa, < \lambda)$-sequence.
\end{proposition}

\begin{proof}
  Suppose for sake of contradiction that $D$ is a thread through $\vec{C}$. Let 
  $A = \acc(D)$. Since $\vec{C}$ is a $\boxtimes^-(\kappa, < \lambda)$-sequence, there is 
  $\beta < \kappa$ such that $\sup(\nacc(C) \cap A) = \beta$ for all $C \in \mathcal{C}_\beta$. Fix 
  such an $\beta$. Then $\beta \in \acc(D)$, so there is $C \in \mathcal{C}_\beta$ 
  such that $C = D \cap \beta$. But then 
  $\nacc(C) \cap A = \beta \cap \nacc(D) \cap \acc(D) = \emptyset$, which 
  is a contradiction.
\end{proof}

Brodsky and Rinot also introduce a further strengthening of $\boxtimes^-_\theta(\mathcal{S}, < \lambda)$. Before we give its 
definition, we need some notation.

\begin{definition}[Brodsky-Rinot, \cite{brodsky_rinot}]
  Suppose $D$ is a set of ordinals and $\sigma$ is an ordinal. $\mathrm{succ}_\sigma(D)$ 
  is, intuitively, the set of the first $\sigma$ ``successor elements'' of $D$. More 
  precisely, $\mathrm{succ}_\sigma(D) = \{\delta \in D \mid$ for some $j < \sigma, 
  \otp(D \cap \delta) = j+1\}$.
\end{definition}

\begin{definition}[Brodsky-Rinot, \cite{brodsky_rinot}]
  Suppose $\kappa$ is a regular, uncountable cardinal, $\lambda, \theta > 1$ are cardinals, and $\mathcal{S}$ 
  is a non-empty collection of stationary subsets of $\kappa$. 
  $\langle \mathcal{C}_\alpha \mid \alpha < \kappa \rangle$ is a $\boxtimes_\theta(\mathcal{S}, < \lambda)$-sequence if:
  \begin{enumerate}
    \item for all $\alpha < \kappa$, $\mathcal{C}_\alpha$ is a set of clubs in $\alpha$ 
      with $0 < |\mathcal{C}_\alpha| < \lambda$;
    \item for all $\alpha < \kappa$ and all $C \in \mathcal{C}_\alpha$, $\otp(C) \leq \theta$;
    \item for all $\alpha < \beta < \kappa$ and all $C \in \mathcal{C}_\beta$, if $\alpha \in \acc(C)$, then 
      $C \cap \alpha \in \mathcal{C}_\alpha$;
    \item for every sequence $\langle A_i \mid i < \kappa \rangle$ of cofinal subsets 
      of $\kappa$ and every $S \in \mathcal{S}$, there is $\beta \in S$ such that, for all $C \in \mathcal{C}_\beta$ and all $i < \beta$, 
      $\sup(\{\alpha < \beta \mid \mathrm{succ}_\omega(C \setminus \alpha) \subseteq A_i\}) 
      = \beta$.
  \end{enumerate}
  As before, we omit $\theta$ if $\theta = \kappa$ and write $\boxtimes_\theta(S, < \lambda)$ instead of 
  $\boxtimes_\theta(\{S\}, < \lambda)$. $\boxtimes(S, \lambda)$ and $\boxtimes(S)$ are defined in the obvious way.
\end{definition}

\begin{remark}
  In clause (4) of the definitions of $\boxtimes^-_\theta(\mathcal{S}, < \lambda)$ and $\boxtimes_\theta(\mathcal{S}, < \lambda)$, the 
  existence of a single $\beta \in S$ is easily seen to be equivalent to the existence of stationarily 
  many such $\beta \in S$. 
\end{remark}

Brodsky and Rinot use these principles, together with $\diamondsuit(\kappa)$ (which 
follows from $\mathrm{GCH}$ for successor cardinals $\kappa > \omega_1$) to construct $\kappa$-Souslin trees 
with various additional properties. 

\begin{theorem}[Brodsky-Rinot, \cite{brodsky_rinot_two}, \cite{brodsky_rinot}, \cite{brodsky_rinot_reduced}, respectively] \label{brodsky_rinot_thm}
  Suppose $\kappa$ is a regular, uncountable cardinal.
  \begin{enumerate}
    \item $\boxtimes^-(\kappa, < \kappa) + \diamondsuit(\kappa)$ implies the existence of a $\kappa$-Souslin tree.
    \item $\boxtimes^-(\mathrm{NS}_\kappa^+) + \diamondsuit(\kappa)$ implies the existence of a coherent 
      $\kappa$-Souslin tree (see \cite{brodsky_rinot} for the definition of coherence in this setting).
    \item $\boxtimes(\kappa) + \diamondsuit(\kappa)$ implies the existence of a $\kappa$-Souslin tree 
      that contains a $\lambda$-ascent path for every infinite cardinal $\lambda < \kappa$ (see 
      Definition \ref{ascending_path_def} for the definition of a $\lambda$-ascent path).
  \end{enumerate}
\end{theorem}

One salient difference between the constructions of Brodsky and 
Rinot and previous constructions is that the new constructions make no explicit use of a 
non-reflecting stationary subset of $\kappa$, while all known previous $\diamondsuit$-based constructions do. This led 
Rinot to ask the following question.

\begin{question}[Rinot, \cite{assaf}] \label{assaf_question}
  Let $\kappa$ be a regular, uncountable cardinal. Is $\mathrm{GCH} + \boxtimes^{-}(\kappa) + \mathrm{Refl}(\kappa)$ 
  consistent? What about $\mathrm{GCH} + \boxtimes(\kappa) + \mathrm{Refl}(\kappa)$?
\end{question}

Here, we answer the first question in the affirmative and the second question in the negative. Upon 
learning of our affirmative answer, Rinot asked if, in case $\kappa = \mu^+$ where $\mu$ is a 
singular cardinal, we can also 
arrange for $\square^*_\mu$ to hold. This is of interest for two reasons. First, the presence of 
$\square^*_\mu$ allows one to draw stronger consequences from $\boxtimes^-(\kappa)$ (see \cite{brodsky_rinot_two}). 
Second, by a result from \cite{rinot_approachability}, if $\mu$ is strong limit, $2^\mu = \mu^+$, 
and $\square^*_\mu$ holds, then forcing to add a Cohen subset of $\mu^+$ necessarily adds a non-reflecting 
stationary set. Therefore, if $\square^*_\mu$ were to also hold in a model of $\mathrm{GCH} + \boxtimes^-(\kappa) + 
\mathrm{Refl}(\kappa)$, the stationary reflection would be quite fragile.
We answer this affirmatively as well. In particular, 
we prove the following theorems.

\begin{theorem} \label{con_thm}
  Assuming the consistency of certain large cardinals, the following are consistent:
  \begin{enumerate}
    \item $\mathrm{GCH} + \boxtimes^{-}(\mathrm{NS}^+_{\aleph_{\omega + 1}}) 
      + \boxtimes^-_{\aleph_\omega}(\{S^{\aleph_{\omega+1}}_{\aleph_n} \mid n < \omega\}, \omega) + \mathrm{Refl}(\aleph_{\omega + 1})$;
    \item $\mathrm{GCH} + \boxtimes^-(\mathrm{NS}^+_{\aleph_2}) + 
      \mathrm{Refl}(S^{\aleph_2}_{\aleph_0})$;
    \item $\mathrm{GCH} + \diamondsuit(\kappa) + \boxtimes^-(\mathrm{NS}^+_\kappa) + \mathrm{Refl}(\kappa)$, 
      where $\kappa$ is the least inaccessible cardinal.
  \end{enumerate}
\end{theorem}

\begin{remark} \label{remark_1}
  We also get the relevant instance of $\diamondsuit$ in (1) and (2) of Theorem \ref{con_thm}, as it follows there from 
  GCH. In addition, there is nothing special about $\aleph_{\omega + 1}$, 
  $\aleph_2$, or the least inaccessible in Theorem \ref{con_thm}. 
  They are used for concreteness only, and self-evident modifications of the proof will allow the reader to 
  obtain similar results for other successors of singular cardinals, successors of regular cardinals, and 
  inaccessible cardinals, respectively.

  In addition, we will obtain similar results about the consistency of wider $\boxtimes$ sequences together with 
  simultaneous stationary reflection, indicating that the existence of a $\kappa$-Souslin tree is in 
  general compatible with a high degree of simultaneous stationary reflection.
\end{remark}

\begin{theorem} \label{incon_thm}
  Suppose $\kappa$ is a regular, uncountable cardinal and $\boxtimes(\kappa)$ holds. Then 
  $\mathrm{Refl}(S^\kappa_\omega)$ fails.
\end{theorem}

However, we will show that $\boxtimes(\kappa)$ is consistent with $\mathrm{Refl}(\kappa \setminus S)$ 
for some non-reflecting stationary $S \subseteq S^\kappa_\omega$.

A small note is in order here. Very soon after the author proved Theorem \ref{con_thm}, Rinot 
proved in \cite{rinot} that, if $\kappa \geq \aleph_2$ is a successor cardinal, then $\mathrm{GCH} + \square(\kappa)$ 
actually implies $\boxtimes^-(\kappa)$, in which case the positive answer to the first part of 
Question \ref{assaf_question} follows directly from work of Hayut and the author in \cite{hayut_lh}. 
Our results here remain of interest, though, in that they provide a strengthened version of 
$\boxtimes^-(\kappa)$ (namely $\boxtimes^-(\mathrm{NS}^+_\kappa)$) and also work for inaccessible values of $\kappa$.

\subsection{Ascending paths in trees}

The following definition is a generalization of the notion of a cofinal branch through a tree.

\begin{definition} \label{ascending_path_def}
  Suppose $T$ is a tree and $\height(T) = \kappa$. Let $\lambda > 0$ be a cardinal.
  \begin{enumerate}
    \item A sequence $\langle b_\gamma:\lambda \rightarrow T_\gamma \mid \gamma < \kappa \rangle$ 
      is an \emph{ascending path of width $\lambda$ through $T$} if, for all 
      $\alpha < \beta < \kappa$, there are $\eta, \xi < \lambda$ such that 
      $b_\alpha(\eta) <_T b_\beta(\xi)$. 
    \item A sequence $\langle b_\gamma:\lambda \rightarrow T_\gamma 
      \mid \gamma < \kappa \rangle$ is a \emph{$\lambda$-ascent path through $T$} if, 
      for all $\alpha < \beta < \kappa$, there is $\eta < \lambda$ such that, 
      for all $\eta \leq \xi < \lambda$, $b_\alpha(\xi) <_T b_\beta(\xi)$.
  \end{enumerate}
\end{definition}

It is clear that a $\lambda$-ascent path is also an ascending path of width 
$\lambda$. The notion of a $\lambda$-ascent path is due to Laver and stems from his work in 
\cite{laver_shelah}. 

One of the reasons special trees are of interest is that they are robustly branchless, i.e. 
if $\kappa$ is a regular, uncountable cardinal and $T$ is a special tree of height $\kappa$, 
then $T$ fails to have a cofinal branch in any outer model in which $\kappa$ remains a 
regular cardinal. Extending this idea, Shelah, building on work of Laver and Todorcevic, 
proved that the existence of certain ascent paths also precludes a tree from being special.

\begin{theorem} [Shelah, \cite{shelah_stanley}]
  Suppose $\lambda < \mu$ are infinite cardinals such that $\lambda$ is regular 
  and $\cf(\mu) \neq \lambda$. Suppose $T$ is a tree of height $\mu^+$ and 
  $T$ has a $\lambda$-ascent path. Then $T$ is not special.
\end{theorem}

Todorcevic and Torres P\'{e}rez, in \cite{todorcevic_torres}, prove a stronger result 
that is further generalized by L\"{u}cke in \cite{luecke} to show that, in many cases, the weaker requirement that 
$T$ not have an ascending path of narrow width is enough to stop $T$ from being special.

\begin{theorem} [L\"{u}cke, \cite{luecke}]
  Suppose $\lambda < \kappa$ are infinite cardinals such that $\kappa$ is regular and is not 
  the successor of a cardinal $\mu$ such that $\cf(\mu) \leq \lambda$. Suppose $T$ is a tree 
  of height $\kappa$ and $S \subseteq S^\kappa_{>\lambda}$ is a stationary set that is 
  non-stationary with respect to $T$. Then $T$ does not have an ascending path of width $\lambda$.
\end{theorem}

In particular, if $T$ is a tree of height $\kappa$, $\kappa$ is not the successor of a cardinal 
of cofinality $\leq \lambda$, and $T$ has an ascending path of width $\lambda$, then $T$ is not special. 

This leads naturally to the following question.

\begin{question}[L\"{u}cke, \cite{luecke}]
  Is it consistent that there is a singular cardinal $\mu$ and a cardinal $\lambda$ such that 
  $\cf(\mu) \leq \lambda < \mu$ and there is a special tree $T$ of height $\mu^+$ that has an 
  ascending path of width $\lambda$?
\end{question}

We will answer this question affirmatively, in fact showing that a stronger statement follows from 
$\square_\mu$ but not from $\square_{\mu, 2}$. More precisely, we will prove the following.

\begin{theorem} \label{asc_path_thm}
  Suppose $\mu$ is a singular cardinal.
  \begin{enumerate}
    \item If $\square_\mu$ holds, then there is a special $\mu^+$-tree with a $\cf(\mu)$-ascent path.
    \item It is consistent that $\square_{\mu, 2}$ holds and, for every regular $\lambda < \mu$, 
      if $T$ is a tree of height $\mu^+$ with a $\lambda$-ascent path, then $T$ has a cofinal branch.
  \end{enumerate}
\end{theorem}

Clause (2) of Theorem \ref{asc_path_thm} is due to Shani and is proven in \cite{shani}. The proof there goes through an argument 
about fresh subsets of ordinals in ultrapowers. We provide a different proof here, using work of the author 
from \cite{covering_2}. In the model we will construct for clause (2) of Theorem \ref{asc_path_thm} (as well as in that 
constructed by Shani in \cite{shani}), we will have $2^\mu = \mu^+$. In \cite{brodsky_rinot}, 
Brodsky and Rinot show that, if $\mu$ is singular, then $\square_\mu + 2^\mu = \mu^+$ implies 
$\boxtimes(\mu^+) + \diamondsuit(\mu^+)$. Therefore, by clause (3) of Theorem \ref{brodsky_rinot_thm}, 
$\square_\mu + 2^\mu = \mu^+$ implies the existence of a $\mu^+$-Souslin tree that has a $\lambda$-ascent 
path for every infinite cardinal $\lambda < \mu$. This shows that 
Brodsky and Rinot's result is optimal in the sense that the existence of such a Souslin tree does not follow from 
$\square_{\mu, 2} + 2^\mu = \mu^+$.

Versions of clause (2) of Theorem \ref{asc_path_thm} also hold for successors of regular cardinals 
and inaccessible cardinals. This provides an example of a case in which $\square_{\mu, 2}$ is compatible with 
a compactness principle that is denied by $\square_\mu$. As another example of such a result, Sakai shows in 
\cite{sakai} that, if $\mu$ is a regular, uncountable cardinal, then $\square_{\mu, 2}$ is compatible with the 
Chang's Conjecture variation $(\mu^+, \mu) \twoheadrightarrow (\omega_1, \omega)$, whereas results of 
Todorcevic (see \cite{todorcevic}) imply that, for any cardinal $\nu < \mu$, $(\mu^+, \mu) \twoheadrightarrow 
(\nu^+, \nu)$ implies the failure of $\square_\mu$.

\section{Forcing preliminaries} \label{forcing_sect}

In this section, we introduce some forcing posets that will be useful for us. 
We start by looking at an indexed strengthening of $\square_{\mu, \cf(\mu)}$, studied in 
\cite{cfm} and \cite{cummings-schimmerling}.

\begin{definition} [Cummings, Foreman, and Magidor \cite{cfm}]
  Suppose $\mu$ is a singular cardinal. A sequence $\langle C_{\alpha, i} \mid 
  \alpha < \mu^+, i(\alpha) \leq i < \cf(\mu) \rangle$ is a $\square^{\mathrm{ind}}_{\mu, 
  \cf(\mu)}$-sequence if the following conditions hold:
  \begin{enumerate}
    \item for every $\alpha < \mu^+$, $i(\alpha) < \cf(\mu)$;
    \item for every $\alpha < \mu^+$ and $i(\alpha) \leq i < \cf(\mu)$, 
      $C_{\alpha, i}$ is club in $\alpha$;
    \item there is an increasing sequence of regular cardinals, $\langle \mu_i 
      \mid i < \cf(\mu) \rangle$, cofinal in $\mu$, such that, for all 
      $\alpha < \mu^+$ and $i(\alpha) \leq i < \cf(\mu)$, $\otp(C_{\alpha, i}) < \mu_i$;
    \item for all $\alpha < \mu^+$ and $i(\alpha) \leq i < j < \cf(\mu)$, $C_{\alpha, i} 
      \subseteq C_{\alpha, j}$;
    \item for all $\alpha < \beta < \mu^+$ and $i(\beta) \leq i < \cf(\mu)$, if 
      $\alpha \in \acc(C_{\beta, i})$, then $i(\alpha) \leq i$ and $C_{\beta, i} 
      \cap \alpha = C_{\alpha, i}$;
    \item for all limit $\alpha < \beta < \mu^+$, there is $i < \cf(\mu)$ such that 
      $i(\alpha), i(\beta) \leq i$ and $\alpha \in \acc(C_{\beta, i})$.
  \end{enumerate}
  $\square^{\mathrm{ind}}_{\mu, \cf(\mu)}$ holds if there is a 
  $\square^{\mathrm{ind}}_{\mu, \cf(\mu)}$-sequence.
\end{definition}

$\square^{\mathrm{ind}}_{\mu, \cf(\mu)}$ can be introduced by a natural forcing poset.

\begin{definition}
  Suppose $\mu$ is a singular cardinal and $\vec{\mu} = \langle \mu_i \mid i < \cf(\mu) \rangle$ 
  is an increasing sequence of regular cardinals, cofinal in $\mu$. Then $\bb{S}^{\mathrm{ind}}_{\vec{\mu}}$ 
  is the forcing poset whose conditions are all $s = \langle C^s_{\alpha, i} \mid \alpha \leq \gamma^s, 
  i(\alpha)^s \leq i < \cf(\mu) \rangle$ such that:
  \begin{enumerate}
    \item $\gamma^s < \mu^+$;
    \item for all $\alpha \leq \gamma^s$, $i(\alpha)^s < \cf(\mu)$;
    \item for all $\alpha \leq \gamma^s$ and $i(\alpha)^s \leq i < \cf(\mu)$, 
      $C^s_{\alpha, i}$ is club in $\alpha$ and $\otp(C^s_{\alpha, i}) < \mu_i$;
    \item for all $\alpha \leq \gamma^s$ and $i(\alpha)^s \leq i < j < \cf(\mu)$, 
      $C^s_{\alpha, i} \subseteq C^s_{\alpha, j}$;
    \item for all $\alpha < \beta \leq \gamma^s$ and $i(\beta)^s \leq i < \cf(\mu)$, 
      if $\alpha \in \acc(C^s_{\beta, i})$, then $i(\alpha)^s \leq i$ and 
      $C^s_{\beta, i} \cap \alpha = C^s_{\alpha, i}$;
    \item for all limit $\alpha < \beta \leq \gamma^s$, there is $i < \cf(\mu)$ 
      such that $i(\alpha)^s, i(\beta)^s \leq i$ and $\alpha \in \acc(C^s_{\beta, i})$.
  \end{enumerate}
  $\bb{S}^{\mathrm{ind}}_{\vec{\mu}}$ is ordered by end-extension.
\end{definition}

The following lemma is proven in Section 9 of \cite{cfm}.

\begin{lemma}
  Suppose $\mu$ and $\vec{\mu}$ are as in the previous definition.
  \begin{enumerate}
    \item $\bb{S}^{\mathrm{ind}}_{\vec{\mu}}$ is $\cf(\mu)$-directed closed.
    \item $\bb{S}^{\mathrm{ind}}_{\vec{\mu}}$ is $<\mu$-strategically closed.
    \item If $G$ is $\bb{S}^{\mathrm{ind}}_{\vec{\mu}}$-generic over $V$, then, in $V[G]$, 
      $\bigcup G$ is a $\square^{\mathrm{ind}}_{\mu, \cf(\mu)}$-sequence.
  \end{enumerate}
\end{lemma}

There is also a natural forcing notion to add a $\square(\kappa)$-sequence.

\begin{definition}
  Suppose $\kappa$ is a regular, uncountable cardinal. $\bb{S}(\kappa)$ is the 
  forcing poset whose conditions are all $s = \langle D^s_\alpha \mid \alpha \leq \gamma^s \rangle$ 
  such that:
  \begin{enumerate}
    \item $\gamma^s < \kappa$;
    \item for all $\alpha \leq \gamma^s$, $D^s_\alpha$ is club in $\alpha$;
    \item for all $\alpha < \beta \leq \gamma^s$, if $\alpha \in \acc(D_\beta)$, 
      then $D_\beta \cap \alpha = D_\alpha$.
  \end{enumerate}
  $\bb{S}(\lambda)$ is ordered by end-extension.
\end{definition}

The following lemma is standard. A proof can be found in \cite{covering}.

\begin{lemma}
  Suppose $\kappa$ is a regular, uncountable cardinal.
  \begin{enumerate}
    \item $\bb{S}(\kappa)$ is countably closed.
    \item $\bb{S}(\kappa)$ is $\kappa$-strategically closed.
    \item If $G$ is $\bb{S}(\kappa)$-generic over $V$, then, in $V[G]$, 
      $\bigcup G$ is a $\square(\kappa)$-sequence.
  \end{enumerate}
\end{lemma}

We now introduce posets designed to add threads to these square sequences.

\begin{definition}
  Suppose $\mu$ is a singular cardinal and $\vec{C} = \langle C_{\alpha, i} \mid 
  \alpha < \mu^+, i(\alpha) \leq i < \cf(\mu) \rangle$ is a 
  $\square^{\mathrm{ind}}_{\mu, \cf(\mu)}$-sequence. Let $i < \cf(\mu)$. $\bb{T}_{0, i}(\vec{C})$ 
  is the forcing poset whose conditions are all $C_{\alpha, i}$ such that $\alpha < \mu^+$ 
  is a limit ordinal and $i(\alpha) \leq i$. $\bb{T}_{0, i}(\vec{C})$ is ordered by end-extension, i.e. 
  $C_{\beta, i} \leq C_{\alpha, i}$ iff $\alpha \in \acc(C_{\beta, i})$.
\end{definition}

\begin{definition}
  Suppose $\kappa$ is a regular, uncountable cardinal and $\vec{D} = \langle D_\alpha \mid 
  \alpha < \kappa \rangle$ is a $\square(\kappa)$-sequence. $\bb{T}_1(\vec{D})$ is the 
  forcing poset whose conditions are all $D_\alpha$ such that $\alpha < \kappa$ is a limit 
  ordinal. $\bb{T}_1(\vec{D})$ is ordered by end-extension.
\end{definition}

\begin{lemma} \label{dense_closed}
  Let $\mu$ be a singular cardinal, let $\kappa = \mu^+$, and let $\vec{\mu} = 
  \langle \mu_i \mid i < \cf(\mu) \rangle$ be an increasing sequence of regular 
  cardinals, cofinal in $\mu$. Let $\bb{S}_0 = \bb{S}^{\mathrm{ind}}_{\vec{\mu}}$ 
  and $\bb{S}_1 = \bb{S}(\kappa)$. Let $\dot{\vec{C}}$ be a name for the 
  $\square^{\mathrm{ind}}_{\mu, \cf(\mu)}$-sequence added by $\bb{S}_0$, and let 
  $\dot{\vec{D}}$ be a name for the $\square(\kappa)$-sequence added by $\bb{S}_1$. 
  Let $i < \cf(\mu)$, let $\dot{\bb{T}}_{0, i}$ be a name for $\bb{T}_{0, i}(\vec{C})$, 
  and let $\dot{\bb{T}}_1$ be a name for $\bb{T}_1(\vec{D})$. In $V$, 
  $(\bb{S}_0 \times \bb{S}_1) * (\dot{\bb{T}}_{0,i} \times \dot{\bb{T}}_1)$ has a dense  
  $\mu_i$-directed closed subset.
\end{lemma}

\begin{proof}
  Let $\bb{U}$ be the set of $((s_0, s_1),(\dot{t}_0, \dot{t}_1)) \in (\bb{S}_0 \times \bb{S}_1) * 
  (\dot{\bb{T}}_{0,i} \times \dot{\bb{T}}_1)$ such that:
  \begin{itemize}
    \item $\gamma^{s_0} = \gamma^{s_1} =: \gamma$;
    \item $i(\gamma)^{s_0} \leq i$;
    \item $s_0 \Vdash ``\dot{t}_0 = C^{s_0}_{\gamma, i}"$;
    \item $s_1 \Vdash ``\dot{t}_1 = D^{s_1}_\gamma"$.
  \end{itemize}
  The verification that $\bb{U}$ is dense and $\mu_i$-directed closed is a 
  straightforward combination of the proofs of Lemma 9.6 from \cite{cfm} 
  and Proposition 3.11 from \cite{covering}.
\end{proof}

We will also need some basic facts about the approachability ideal $I[\kappa]$ 
for regular, uncountable $\kappa$. The reader is referred to \cite{eisworth} 
for a wealth of information about $I[\kappa]$. Relevant to us is the fact 
that $I[\kappa]$ is a normal ideal on $\kappa$ extending the non-stationary ideal 
and the following fact, due to Shelah.

\begin{fact} \label{ap_fact}
  Suppose $\lambda < \kappa$ are regular cardinals, and suppose 
  $S \subseteq S^\kappa_{<\lambda}$ is stationary and $S \in I[\kappa]$. Suppose 
  moreover that $\bb{P}$ is a $\lambda$-closed forcing notion. Then $S$ remains 
  stationary in $V^{\bb{P}}$.
\end{fact}

If $\mu$ is an uncountable cardinal, then $\mathrm{AP}_\mu$ is the assertion 
that $\mu^+ \in I[\mu^+]$.

\section{Souslin trees with stationary reflection} \label{con_sect}

In this section, we prove Theorem \ref{con_thm}. We prove (1) in detail 
and indicate how to modify the argument for (2) and (3).

Suppose that, in a model $V_0$ of ZFC, $\langle \lambda_n \mid n < \omega \rangle$ is an increasing 
sequence of cardinals such that $\lambda_0 = \aleph_0$ and, for all $0 < n < \omega$, 
$\lambda_n$ is supercompact. Assume $\mathrm{GCH}$ holds. Let $\mu = \sup(\{\kappa_n \mid n < \omega\})$, and let 
$\kappa = \mu^+$. Define a forcing iteration $\langle \bb{P}_m, \dot{\bb{Q}}_n
\mid m \leq \omega, n < \omega \rangle$, taken with full supports, by letting, for all 
$n < \omega$, $\dot{\bb{Q}}_n$ be a $\bb{P}_n$-name for $\mathrm{Coll}(\lambda_n, < \lambda_{n+1})$. 
Let $\bb{P} = \bb{P}_\omega$, let $G$ be $\bb{P}$-generic over $V_0$, and let $V = V_0[G]$. For $n < \omega$, 
let $\dot{\bb{P}}^n$ be such that $\bb{P} \cong \bb{P}_n * \dot{\bb{P}}^n$ and let $G_n$ and $G^n$ be the 
generic filters induced by $G$ on $\bb{P}_n$ and $\bb{P}^n$, respectively.
In $V$, we have the following situation:
\begin{itemize}
  \item $\lambda_n = \aleph_n$ for all $n < \omega$;
  \item $\mu = \aleph_\omega$;
  \item $\kappa = \aleph_{\omega + 1}$;
  \item $\mathrm{GCH}$ holds;
  \item $\mathrm{AP}_{\aleph_\omega}$ holds.
\end{itemize}

The first four items follow from standard arguments. For the proof of $\mathrm{AP}_{\aleph_\omega}$, 
see \cite{magidor}.

\begin{lemma} \label{indestructibility_lemma}
  Suppose $n < \omega$ and, in $V$, $\bb{U}$ is an $\aleph_{n+1}$-directed closed 
  forcing poset. Then, in $V^{\bb{U}}$, $\mathrm{Refl}(S^{\kappa}_{<\aleph_n})$ holds.
  In fact, for every stationary $S \subseteq S^\kappa_{<\aleph_n}$, $S$ reflects 
  at an ordinal $\beta \in S^\kappa_{\aleph_n}$.
\end{lemma}

\begin{proof}
  Let $j:V_0 \rightarrow M$ witness that $\lambda_{n+1}$ is $|\bb{P} * \dot{\bb{U}}|$-supercompact. 
  Since $|\bb{P}_n| < \lambda_{n+1}$, we can lift $j$ to $j:V_0[G_n] \rightarrow M[G_n]$. 
  $j(\bb{Q}_n) = \mathrm{Coll}(\lambda_n, < j(\lambda_{n+1}))$. By standard arguments, a lemma from \cite{magidor} implies that 
  $j(\bb{Q}_n) \cong \bb{P}^n * \dot{\bb{U}} * \dot{\bb{R}}$, where $\dot{\bb{R}}$ is a name for 
  a $\lambda_n$-closed forcing poset. Thus, letting $H$ be $\bb{U}$-generic over $V_0[G]$ and 
  $I$ be $\bb{R}$-generic over $V_0[G*H]$, we can lift $j$ further to $j:V_0[G_{n+1}] 
  \rightarrow M[G*H*I]$. Since $j(\bb{P}^{n+1}*\dot{\bb{U}})$ is $j(\kappa_{n+1})$-directed closed 
  in $M[G*H*I]$, we can find a lower bound $(q^*, u^*) \in j(\bb{P}^{n+1})$ to $\{j((q,u)) \mid (q,u) 
    \in G^{n+1}*H \}$. Then, letting $G^+ * H^+$ be $j(\bb{P}^{n+1}*\dot{\bb{U}})$-generic over 
  $V_0[G*H*I]$ with $(q^*, u^*) \in G^+ * H^+$, we can extend $j$ once more to 
  $j:V_0[G*H] \rightarrow M[G*H*I*G^+*H^+]$.

  Suppose for sake of contradiction that, in $V_0[G*H]$, $S$ is a stationary subset of 
  $S^\kappa_{<\aleph_n}$ that does not reflect at any ordinal in $S^\kappa_{\aleph_n}$. 
  Then, in $M[G*H*I*G^+*H^+]$, $j(S)$ is 
  a stationary subset of $j(\kappa)$ that does not reflect at any ordinal in $S^{j(\kappa)}_{\aleph_n}$. 
  In particular, if $\eta = \sup(j``\kappa)$, then, in $M[G*H*I*G^+*H^+]$, $\cf(\eta) = \aleph_n$, 
  so $j(S) \cap \eta$ is non-stationary in $\eta$. Let $D$ be club in $\eta$ such 
  that $\otp(D) = \lambda_n$ and $D \cap j(S) = \emptyset$. Let $E = 
  \{\alpha < \kappa \mid j(\alpha) \in D\}$. Then $E \cap S = \emptyset$ and, since 
  $j$ is continuous at points of cofinality $<\lambda_n$, $E$ is club in $\kappa$. 
  Thus, $S$ is non-stationary in $V_0[G*H*I*G^+*H^+]$. In $V_0[G*H]$, $S$ is a stationary 
  subset of $S^\kappa_{<\aleph_n}$. In $V_0[G]$, $S^\kappa_{<\aleph_n} \in I[\kappa]$. 
  Since $H$ does not change any cofinalities below $\aleph_n$, we still have 
  $S^\kappa_{\aleph_n} \in I[\kappa]$ in $V_0[G*H]$. In particular, $S \in I[\kappa]$. 
  Moreover, $I*G^+*H^+$ is generic for $\lambda_n$-closed forcing, so, by Fact \ref{ap_fact}, $S$ remains 
  stationary in $V_0[G*H*I*G^+*H^+]$, which is a contradiction.
\end{proof}

Work now in $V$. Let $\vec{\mu} = \langle \aleph_{i+1} \mid i < \omega \rangle$, let 
$\bb{S}_0 = \bb{S}^{\mathrm{ind}}_{\vec{\mu}}$, and let $\bb{S}_1 = \bb{S}(\kappa)$. 
Let $\bb{S} = \bb{S}_0 \times \bb{S}_1$.

In $V^\bb{S}$, let $\vec{C} = \langle C_{\alpha, i} \mid \alpha < \kappa, i(\alpha) 
\leq i < \omega \rangle$ be the generic $\square^{\mathrm{ind}}_{\mu, \omega}$-sequence 
introduced by $\bb{S}_0$, and let $\vec{D} = \langle D_\alpha \mid \alpha < \kappa \rangle$ 
be the generic $\square(\kappa)$-sequence introduced by $\bb{S}_1$. For $i < \omega$, 
let $\bb{T}_{0, i} = \bb{T}_{0, i}(\vec{C})$, let $\bb{T}_1 = \bb{T}_1(\vec{D})$, and let 
$\bb{T}(i) = \bb{T}_{0, i} \times \bb{T}_1$. By Lemma \ref{dense_closed}, 
$\bb{S} * \dot{\bb{T}}(i)$ has a dense $\aleph_{i+1}$-directed closed subset.

The proof of the following lemma is as in Lemma 9.8 of \cite{cfm}.

\begin{lemma}
  Let $i < j < \omega$. In $V^\bb{S}$, define a map $\pi_{ij}:\bb{T}(i) \rightarrow 
  \bb{T}(j)$ by letting, for all $(C_{\alpha, i}, D_\beta) \in \bb{T}(i)$, 
  $\pi_{ij}((C_{\alpha,i}, D_\beta)) = (C_{\alpha, j}, D_\beta)$. Then $\pi_{ij}$ 
  is a projection.
\end{lemma}

\begin{definition}
  In $V^{\bb{S}}$ or any forcing extension thereof, we say that a subset $S \subseteq \kappa$ 
  is \emph{fragile} if, for all $i < \omega$, $\Vdash_{\bb{T}(i)}``S$ is non-stationary in 
  $\kappa."$
\end{definition}

\begin{remark}
  If $i < j < \omega$, $S \subseteq \kappa$, $t \in \bb{T}(i)$, and 
  $t \Vdash_{\bb{T}(i)} ``S$ is stationary$."$, then, as $\pi_{ij}$ is a projection, 
  $\pi_{ij}(t) \Vdash_{\bb{T}(j)} ``S$ is stationary$."$ Thus, if $S$ is not fragile, 
  then, for all sufficiently large $i < \omega$, there is $t \in \bb{T}(i)$ such that 
  $t \Vdash_{\bb{T}(i)}``S$ is stationary$."$
\end{remark}

In $V^{\bb{S}}$, recursively define posets $\langle \bb{R}_\eta \mid \eta \leq \kappa^+ \rangle$
and names $\langle \dot{S}_\eta \mid \eta < \kappa^+ \rangle$ such that:
\begin{enumerate}
  \item for all $\eta < \kappa^+$, $\dot{S}_\eta$ is an $\bb{R}_\eta$-name for a 
    fragile subset of $\kappa$;
  \item for all $\eta < \kappa^+$, conditions of $\bb{R}_\eta$ are all functions $r$ such that:
    \begin{enumerate}
      \item $\dom(r) \subseteq \eta$;
      \item $|\dom(r)| \leq \mu$;
      \item for all $\xi \in \dom(r)$, $r(\xi)$ is a closed, bounded subset of $\kappa$ 
        and $r \restriction \xi \Vdash_{\bb{R}_\xi} ``r(\xi) \cap \dot{S}_\xi = \emptyset"$;
    \end{enumerate}
  \item for all $\eta < \kappa^+$, if $r_0, r_1 \in \bb{R}_\eta$, then $r_1 \leq r_0$ if:
    \begin{enumerate}
      \item $\dom(r_0) \subseteq \dom(r_1)$;
      \item for all $\xi \in \dom(r_0)$, $r_1(\xi)$ end-extends $r_0(\xi)$.
    \end{enumerate}
\end{enumerate}

Let $\bb{R} = \bb{R}_{\kappa^+}$.
We will show that each $\bb{R}_\eta$ is $\kappa$-distributive and, therefore, for all 
$\eta < \kappa^+$, $\bb{R}_{\eta + 1} \cong \bb{R}_\eta * \mathrm{CU}(\kappa \setminus \dot{S}_\eta)$, 
where $\mathrm{CU}(T)$ is the standard forcing poset for shooting a club through $T$. 
Also, standard arguments show that $\bb{R}$ has the $\kappa^+$-c.c. Therefore, by 
employing a sufficient bookkeeping apparatus in our choice of $\langle \dot{S}_\eta \mid \eta < \kappa^+ \rangle$, 
we may arrange so that, in $V^{\bb{S} * \dot{\bb{R}}}$, 
for all $T \subseteq \kappa$, if $T$ is fragile, then $T$ is non-stationary.

Moving back to $V$, for each $\eta < \kappa^+$ and $i < \omega$, let $\dot{E}_{\eta, i}$ be 
an $\bb{S} * \dot{\bb{R}}_\eta * \dot{\bb{T}}(i)$-name for a club in $\kappa$ disjoint 
from $\dot{S}_\eta$. If $i < j < \omega$, then, since $\pi_{ij}$ is a projection from $\bb{T}(i)$ to 
$\bb{T}(j)$ in $V^{\bb{S}}$, we may also consider $\dot{E}_{\eta, j}$ as an 
$\bb{S} * \dot{\bb{R}}_\eta * \dot{\bb{T}}(i)$-name and assume that 
$\Vdash_{\bb{S} * \dot{\bb{R}}_\eta * \dot{\bb{T}}(i)}``\dot{E}_{\eta, i} \subseteq \dot{E}_{\eta, j}."$

\begin{lemma} \label{dense_closed_2}
  For all $i < \omega$ and $\eta \leq \kappa^+$, $\bb{S} * \dot{\bb{R}}_\eta * \dot{\bb{T}}(i)$ has a 
  dense $\aleph_{i+1}$-directed closed subset.
\end{lemma}

\begin{proof}
  The ideas of this proof are largely derived from the ideas in Section 10 of \cite{cfm}. For 
  sake of completeness and because we have simplified some aspects of the arguments in \cite{cfm}, 
  we present the proof in some detail.

  For $i < \omega$, let $\bb{U}_{0,i}$ be the dense $\aleph_{i+1}$-directed closed subset of $\bb{S} * \dot{\bb{T}}(i)$ 
  given in the proof of Lemma \ref{dense_closed}. For $\eta < \kappa^+$ and $i < \omega$, let 
  $\bb{U}_{\eta, i}$ be the set of $(s, \dot{r}, \dot{t}) \in \bb{S} * \dot{\bb{R}}_\eta 
  * \dot{\bb{T}}(i)$ such that:
  \begin{itemize}
    \item $(s, \dot{t}) \in \bb{U}_{0,i}$;
    \item $s$ decides the value of $\dot{r}$, i.e. there is a function $r \in V$ such that 
      $s \Vdash ``\dot{r} = \check{r}"$;
    \item for all $\xi \in \dom(r)$, $(s, \dot{r} \restriction \xi, \dot{t}) 
      \Vdash_{\bb{S} * \dot{\bb{R}}_\xi * \dot{\bb{T}}(i)}``\max(r(\xi)) \in \dot{E}_{\xi, i}."$
  \end{itemize}
  The verification that $\bb{U}_{\eta, i}$ is $\aleph_{i+1}$-directed closed is straightforward. It  
  is thus sufficient to show that it is dense. We do this by induction on $\eta$, simultaneously 
  for all $i$.

  If $\eta = 0$, this follows from Lemma \ref{dense_closed}. Suppose $0 < \eta \leq \lambda^+$, $i < \omega$, and 
  we have proven that, for all $\xi < \eta$ and all $j < \omega$, $\bb{U}_{\xi, j}$ is dense in 
  $\bb{S} * \dot{\bb{R}}_\xi * \dot{\bb{T}}(j)$. Fix $(s_0, \dot{r}_0, \dot{t}_0) \in \bb{S} * 
  \dot{\bb{R}}_\eta * \dot{\bb{T}}(i)$. We may assume that $(s_0, \dot{t}_0) \in \bb{U}_{0, i}$ and, 
  since $\bb{S}$ is $\kappa$-distributive, that $s_0$ decides the value of $\dot{r}_0$ to be some 
  $r_0 \in V$. We will find $(s, \dot{r}, \dot{t}) \leq (s_0, \dot{r}_0, \dot{t}_0)$ with 
  $(s, \dot{r}, \dot{t}) \in \bb{U}_{\eta, i}$.

  \textbf{Case 1: $\eta = \xi + 1$.} By the inductive hypothesis, we can find 
  $(s_1, \dot{r}_1, \dot{t}_1) \leq (s_0, \dot{r}_0 \restriction \xi, \dot{t}_0)$ such 
  that $(s_1, \dot{r}_1, \dot{t}_1) \in \bb{U}_{\xi, i}$ and there is $\alpha > \max(r_0(\xi))$ 
  such that $(s_1, \dot{r}_1, \dot{t}_1) \Vdash``\alpha \in \dot{E}_{\eta, i}."$ Now form 
  $(s, \dot{r}, \dot{t})$ by letting $(s, \dot{t}) = (s_1, \dot{t}_1)$ and letting $\dot{r}$ 
  be such that $s \Vdash ``\dot{r} \restriction \xi = \dot{r}_1$ and $\dot{r}(\xi) = r_0(\xi) \cup \{\alpha\}."$
  $(s, \dot{r}, \dot{t})$ is easily seen to be in $\bb{U}_{\eta, i}$.

  \textbf{Case 2: $\cf(\eta) \geq \kappa$.} In this case, $\dom(r)$ is bounded below $\eta$, 
  so there is some $\xi < \eta$ such that $(s_0, \dot{r}_0, \dot{t}_0) \in \bb{S} * \dot{\bb{R}}_\xi 
  * \dot{\bb{T}}(i)$, and we may simply invoke the inductive hypothesis for $\xi$.

  \textbf{Case 3: $\aleph_0 \leq \cf(\eta) < \mu$.} Let $j < \omega$ be such that $i \leq j$ 
  and $\cf(\eta) < \aleph_j$. Let $\gamma_0 = \gamma^{s_0}$, and note that 
  $s_0 \Vdash ``\dot{t}_0 = C^{s_0}_{\gamma_0, i}."$ Let $\dot{t}^*_0$ be such that 
  $s_0 \Vdash ``\dot{t}^*_0 = C^{s_0}_{\gamma_0, j}$ and note that $(s_0, \dot{r}_0, \dot{t}^*_0) 
  \in \bb{S}* \dot{\bb{R}}_\eta * \dot{\bb{T}}(j)$. Let $\langle \xi_k \mid k < \cf(\eta) \rangle$ 
  be an increasing, continuous sequence of ordinals, cofinal in $\eta$. Recursively construct a 
  sequence of conditions $\langle (s^k, \dot{r}^k, \dot{t}^k) \mid k < \cf(\eta) \rangle$ such that:
  \begin{itemize}
    \item for all $k < \cf(\eta)$, $(s^k, \dot{r}^k, \dot{t}^k) \in \bb{U}_{\xi_k, j}$;
    \item for all $k < \cf(\eta)$, $(s^k, \dot{r}^k, \dot{t}^k) \leq (s_0, \dot{r}_0 \restriction 
      \xi_k, \dot{t}^*_0)$;
    \item for all $k_0 < k_1 < \cf(\eta)$, $(s^{k_1}, \dot{r}^{k_1}, \dot{t}^{k_1}) \leq 
      (s^{k_0}, \dot{r}^{k_0}, \dot{t}^{k_0})$;
  \end{itemize}
  The construction is straightforward by the inductive hypothesis and the closure of the relevant posets. 
  Now $\langle (s^k, \dot{r}^k, \dot{t}^k) \mid k < \cf(\eta) \rangle$ is a decreasing sequence 
  in $\bb{U}_{\eta, j}$, so, by the closure of $\bb{U}_{\eta, j}$, we may find a lower bound, 
  $(s^*, \dot{r}, \dot{t}^*) \in \bb{U}_{\eta, j}$. Now, using the fact that $\pi_{ij}$ is a projection 
  in $V^{\bb{S}}$, we may find $(s, \dot{t}) \in \bb{U}_{0, i}$ such that $(s, \dot{t}) \leq (s_0, \dot{t}_0)$ 
  and $(s, \pi_{ij}(\dot{t})) \leq (s^*, \dot{t}^*)$. Now, using the fact that, for all $\xi < \eta$,
  $\Vdash_{\bb{S} * \dot{\bb{R}}_\xi * \dot{\bb{T}}(i)}``\dot{E}_{\xi, i} \subseteq \dot{E}_{\xi, j},"$ we 
  have that $(s, \dot{r}, \dot{t}) \in \bb{U}_{\eta, i}$ and $(s, \dot{r}, \dot{t}) \leq (s_0, \dot{r}_0, \dot{t}_0)$.
\end{proof}

Let $H$ be $\bb{S}$-generic over $V$, and let $I$ be $\bb{R}$-generic over $V[H]$. We now argue that 
$V[H*I]$ satisfies the requirements of (1) in Theorem \ref{con_thm}. It is easily seen that, as GCH holds in $V$, 
it holds in $V[H*I]$ as well. We must therefore verify 
$\mathrm{Refl}(\aleph_{\omega + 1})$, $\boxtimes^{-}(\mathrm{NS}^+_{\aleph_{\omega+1}})$, 
and $\boxtimes^-_{\mu}(\{S^{\kappa}_{\aleph_n} \mid n < \omega\}, \omega)$.

\begin{lemma} \label{reflection_lem}
  $\mathrm{Refl}(\kappa)$ holds in $V[H*I]$. In fact, for every $i \leq j < \omega$ 
  and every stationary $S \subseteq S^\kappa_{<\aleph_i}$, $S$ reflects at an ordinal in $S^\kappa_{\aleph_j}$.
\end{lemma}

\begin{proof}
  Since, for every stationary $S \subseteq \kappa$, there is $i < \omega$ such that $S \cap S^\kappa_{<\aleph_i}$ 
  is stationary, it suffices to prove the second statement. Thus, fix $i \leq j < \omega$ and a 
  stationary $S \subseteq S^\kappa_{<\aleph_i}$. By the construction of $\bb{R}$, since $S$ is 
  stationary in $V[H*I]$, $S$ is not fragile. Therefore, we may find $\ell > \max(i, j)$ and 
  $t \in \bb{T}(\ell)$ such that $t \Vdash ``S$ is stationary$."$ Let $J$ be $\bb{T}(\ell)$-generic 
  over $V[H*I]$ with $t \in J$. $S$ is thus stationary in $V[H*I*J]$. By Lemma \ref{dense_closed_2}, 
  $V[H*I*J]$ can be viewed as a forcing extension of $V$ by an $\aleph_{\ell + 1}$-directed closed 
  forcing notion. Therefore, by Lemma \ref{indestructibility_lemma}, since $\ell > j$, $S$ reflects in $V[H*I*J]$ 
  at an ordinal in $S^\kappa_{\aleph_j}$. Since $V[H*I]$ and $V[H*I*J]$ have the same ordinals of 
  cofinality $\aleph_j$, this holds in $V[H*I]$ as well.
\end{proof}

Recall that $\vec{D} = \langle D_\alpha \mid \alpha < \kappa \rangle$ is the generic 
$\square(\kappa)$-sequence introduced by $\bb{S}$. An easy genericity argument, which we
omit, yields the following proposition.

\begin{proposition} \label{genericity_prop}
  Suppose $\gamma < \delta < \kappa$, with $\gamma$ a limit ordinal. Then there is 
  $\beta < \kappa$ such that $\gamma \in \acc(D_\beta)$ and $\delta \in \nacc(D_\beta)$.
\end{proposition}

\begin{lemma}
  In $V[H*I]$, $\vec{D}$ is a $\boxtimes^-(\mathrm{NS}^+_\kappa)$-sequence.
\end{lemma}

\begin{proof}
  Suppose not. This means that there is a cofinal $A \subseteq \kappa$ 
  and a stationary $S \subseteq \kappa$ such that, for all $\alpha \in S$, 
  $\sup(\nacc(D_\alpha) \cap A) < \alpha$. By Fodor's Lemma, we can find a fixed 
  $\alpha_0 < \kappa$ and a stationary $S_0 \subseteq S$ such that, for all 
  $\alpha \in S_0$, $\sup(\nacc(D_\alpha) \cap A) = \alpha_0$. $S_0$ is not fragile, 
  so we can find $i < \omega$ and $t \in \bb{T}(i)$ such that $t \Vdash``S_0$ is 
  stationary in $\kappa."$ Let $\gamma < \kappa$ be such that, letting $t = (t_0, t_1)$, 
  we have $t_1 = D_\gamma$. Let $\delta = \min(A \setminus (\gamma + 1))$. By 
  Proposition \ref{genericity_prop}, there is $\beta < \kappa$ such that 
  $\gamma \in \acc(D_\beta)$ and $\delta \in \nacc(D_\beta)$. Let $t_1^* = D_\beta$, 
  and let $t^* = (t_0, t_1^*) \leq t$. Let $J$ be $\bb{T}(i)$-generic with 
  $t^* \in J$. Let $T = \bigcap_{(\bar{t}_0, \bar{t}_1) \in J} \bar{t}_1$. $T$ is a 
  thread through $\vec{D}$ and $\beta \in \acc(T)$. Therefore, for every 
  $\alpha \in \acc(T) \setminus \beta$, $\delta \in \nacc(D_\alpha) \cap A$, so, 
  in particular, $\alpha \not\in S_0$. Thus, $\acc(T) \setminus \beta$ witnesses 
  that $S_0$ is non-stationary in $\kappa$, contradicting the fact that 
  $t \Vdash``S_0$ is stationary in $\kappa."$ and $t \in J$.
\end{proof}

\begin{lemma} \label{higher_width_lemma}
  In $V[H*I]$, $\vec{\mathcal{C}}$ is a $\boxtimes^-_{\mu}(\{S^{\kappa}_{\aleph_n} \mid n < \omega\}, \omega)$-sequence.
\end{lemma}

\begin{proof}
  We first deal with $S^\kappa_{\aleph_0}$. Work in $V$. Let $\dot{A}$ be an $\bb{S} * \dot{\bb{R}}$-name for an unbounded 
  subset of $\kappa$, and let $(s, \dot{r}) \in \bb{S} \times \dot{\bb{R}}$, 
  with $s = (s(0), s(1))$. We will find $(s^*, \dot{r}^*) \leq (s, \dot{r})$ 
  and $\beta \in S^\kappa_\omega$ such that $(s^*, \dot{r}^*) \Vdash_{\bb{S} * \dot{\bb{R}}}``$for all $\dot{i}(\beta) \leq i < \omega$, 
  $\sup(\nacc(\dot{C}_{\beta, i}) \cap \dot{A}) = \beta."$

  Let $\bb{U} = \bb{U}_{\kappa^+, 0}$, the countably closed dense subset 
  of $\bb{S} * \dot{\bb{R}} * \dot{\bb{T}}(0)$ isolated in the proof of Lemma 
  \ref{dense_closed_2}. We will recursively define a decreasing sequence $\langle 
  (s_n, \dot{r}_n, \dot{t}_n) \mid n < \omega \rangle$ from $\bb{U}$ and 
  an increasing sequence of ordinals $\langle \alpha_n \mid n < \omega \rangle$.

  To start, fix $(s_0, \dot{r}_0, \dot{t}_0) \in \bb{U}$ such that 
  $(s_0, \dot{r}_0) \leq (s, \dot{r})$. Next, suppose $n < \omega$ and 
  $(s_n, \dot{r}_n, \dot{t}_n)$ has been defined and $r_n$ is the function 
  in $V$ such that $s_n \Vdash_{\bb{S}}``\dot{r}_n = \check{r}_n."$ Assume we have arranged 
  that, for all $m < n$, $\alpha_m$ has also been defined and 
  $\alpha_m < \gamma^{s_n(0)}$. Find $(s^*_n, \dot{r}^*_n) \leq (s_n, \dot{r}_n)$ 
  and $\alpha_n > \gamma^{s_n(0)}$ such that $(s^*_n, \dot{r}^*_n) \Vdash_{\bb{S} * \dot{\bb{R}}} 
  ``\alpha_n \in \dot{A}."$ Let $\gamma^* = \gamma^{s^*_n(0)}$, and let $\gamma = \gamma^*+\omega$.
  We now define $\hat{s}_n(0) \leq s^*_n(0)$ with $\gamma^{\hat{s}_n(0)} = \gamma$. To do 
  this, we only need to specify $i(\gamma)^{\hat{s}_n(0)}$ and $C^{\hat{s}_n(0)}_{\gamma, i}$ 
  for all $i(\gamma)^{\hat{s}_n(0)} \leq i < \omega$. We let $i(\gamma)^{\hat{s}_n(0)} = 0$. 
  Let $m < \omega$ be least such that $m \geq n$ and $\gamma^{s_n(0)} \in \acc(C^{s^*_n(0)}_{\gamma^*, m})$. 
  If $i \leq m$, then let $C^{\hat{s}_n(0)}_{\gamma, i} = C^{s_n(0)}_{\gamma^{s_n(0)}, i} 
  \cup \{\gamma^{s_n(0)}, \alpha_n\} \cup \{\gamma^* + \ell \mid \ell < \omega \}$. 
  If $m < i < \omega$, let $C^{\hat{s}_n(0)}_{\gamma, i} = C^{s^*_n(0)}_{\gamma^*, i} \cup 
  \{\gamma^* + \ell \mid \ell < \omega \}$. The point is that, for all $i \leq n$, 
  $\alpha_n \in \nacc(C^{\hat{s}_n(0)}_{\gamma, i})$.
  
  Let $\hat{s}_n(1) = s^*_n(1)$, and let $\dot{\hat{t}}_n$ be such that $\hat{s}_n \Vdash_{\bb{S}}`` 
  \dot{\hat{t}}_n(0) = C^{\hat{s}_n(0)}_{\gamma, 0}$ and $\dot{\hat{t}}_n(1) = \dot{t}_n(1)."$ 
  Then $(\hat{s}_n, \dot{r}_n, \dot{\hat{t}}_n) \leq (s_n, \dot{r}_n, \dot{t}_n)$. 
  Find $(s_{n+1}, \dot{r}_{n+1}, \dot{t}_{n+1}) \leq (\hat{s}_n, \dot{r}_n, \dot{\hat{t}}_n)$ 
  such that $(s_{n+1}, \dot{r}_{n+1}, \dot{t}_{n+1}) \in \bb{U}$ and continue the construction.

  At the end of the construction, let $(s^*, \dot{r}^*)$ be a lower bound for 
  $\langle s_n, \dot{r}_n \mid n < \omega \rangle$. In particular, we can assume that, 
  letting $\beta = \sup(\{\gamma^{s_n(0)} \mid n < \omega\})$, we have  
  $i(\beta)^{s^*(0)} = 0$ and, for all $i < \omega$, $C^{s^*(0)}_{\beta, i} = 
  \bigcup_{n < \omega} C^{s_n(0)}_{\gamma^{s_n(0)}, i}$. Also, 
  $(s^*, \dot{r}^*) \Vdash_{\bb{S} * \dot{\bb{R}}}``\{\alpha_n \mid n < \omega\} \subseteq \dot{A}"$ 
  and, for all $i < \omega$, $\{\alpha_n \mid i \leq n < \omega\} \subseteq \nacc(C^{s^*(0)}_{\beta, i})$. 
  Therefore, $(s^*, \dot{r}^*) \Vdash_{\bb{S} * \dot{\bb{R}}}``$ for all $i < \omega$, 
  $\sup(\nacc(\dot{C}_{\beta, i}) \cap \dot{A}) = \beta,"$ as desired.

  Now suppose $0 < n < \omega$. In $V[H*I]$, let $A$ be an unbounded subset of $\kappa$. 
  Let $S_A = \{\beta \in S^\kappa_\omega \mid$ for all $i(\beta) \leq i < \omega, \sup(\nacc(C_{\beta,i}) \cap A) = \beta\}$. 
  By the case $n=0$, $S_A$ is stationary, so, by Lemma \ref{reflection_lem}, there is $\gamma \in S^\kappa_{\aleph_n}$ such that 
  $S_A$ reflects at $\gamma$. For every $i(\gamma) \leq i < \omega$, there are 
  unboundedly many $\beta \in \acc(C_{\gamma, i}) \cap S_A$, so $\sup(\nacc(C_{\gamma, i} \cap A) = \gamma$. 
  Therefore, $\vec{\mathcal{C}}$ is a $\boxtimes^-_{\aleph_\omega}(\{S^{\aleph_{\omega+1}}_{\aleph_n} \mid n < \omega\}, \omega)$-sequence.
\end{proof}

This completes the proof of (1) of Theorem \ref{con_thm}. We next quickly sketch 
the proofs of (2) and (3). Suppose that, in $V$, $\mathrm{GCH}$ holds, $\kappa$ is a regular, uncountable cardinal, 
$T$ is a stationary subset of $\kappa$, and $\mathrm{Refl}(T)$ holds in any forcing 
extension by a $\kappa$-directed closed forcing poset of size $\leq \kappa$. By arguments 
from \cite{hayut_lh}, this can be forced with $\kappa = \aleph_2$ and $T = S^{\aleph_2}_{\aleph_0}$ 
from a weakly compact cardinal and with $\kappa$ being the least inacessible cardinal and 
$T = \kappa$ from an inaccessible limit of supercompact cardinals.

Let $\bb{S} = \bb{S}(\kappa)$. In $V^{\bb{S}}$, let $\vec{D}$ be the generically-added 
$\square(\kappa)$-sequence, and let $\bb{T} = \bb{T}_1(\vec{D})$. 
In $V^{\bb{S}}$, define a poset $\bb{R}$ exactly as in the proof of (1) of Theorem \ref{con_thm}. 
In particular, in $V^{\bb{S} * \dot{\bb{R}}}$, if $S \subseteq \kappa$ and 
$\Vdash_{\bb{T}}``S$ is non-stationary$,"$ then $S$ is already non-stationary in 
$V^{\bb{S} * \dot{\bb{R}}}$. We will also have the following Lemma, proven in the same 
way as Lemma \ref{dense_closed_2}.

\begin{lemma}
  In $V$, $\bb{S} * \dot{\bb{R}} * \dot{\bb{T}}$ has a dense $\kappa$-directed closed subset.
\end{lemma}

Let $G$ be $\bb{S}$-generic over $V$, and let $H$ be $\bb{R}$-generic over $V[H]$. The proof 
that, in $V[G*H]$, $\mathrm{Refl}(T)$ holds is similar to the proof in (1) and can be found in 
\cite{hayut_lh}. The proof that $\vec{D}$ is a $\boxtimes^-(\mathrm{NS}^+_\kappa)$-sequence in $V[G*H]$ 
is exactly as in the proof of (1). Thus, by either 
letting $\kappa = \aleph_2$ and $T = S^{\aleph_2}_{\aleph_0}$ or letting $\kappa$ be the 
least inaccessible cardinal and $T = \kappa$, the following lemma will complete the proofs of (2) and (3) 
of Theorem \ref{con_thm}.

\begin{lemma} \label{diamond_lem}
  $\diamondsuit(\kappa)$ holds in $V[G*H]$.
\end{lemma}

\begin{proof}
  In $V[G*H]$, for every $\beta < \kappa$, we define $A_\beta \subseteq \beta$ as follows. 
  For all $\alpha < \beta < \kappa$, let $\alpha \in A_\beta$ iff $\beta + \alpha + 1 \in D_{\beta \cdot 2}$. 
  We claim that, in $V[G*H]$, $\vec{A} = \langle A_\beta \mid \beta < \kappa \rangle$ is a 
  $\diamondsuit(\kappa)$-sequence.

  Work in $V$ and, for all $\beta < \kappa$, let $\dot{A}_\beta$ be a canonical name 
  for $A_\beta$. Let $(s, \dot{r}) \in \bb{S} * \dot{\bb{R}}$, let $\dot{A}$ be an 
  $\bb{S} * \dot{\bb{R}}$-name for an unbounded subset of $\kappa$, and let $\dot{C}$ 
  be an $\bb{S} * \dot{\bb{R}}$-name for a club in $\kappa$. We will find $(s^*, \dot{r}^*) 
  \leq (s, \dot{r})$ and $\beta < \kappa$ such that $(s^*, \dot{r}^*) \Vdash_{\bb{S} * \dot{\bb{R}}} 
  ``\beta \in \dot{C}$ and $\dot{A} \cap \beta = \dot{A}_\beta."$

  Let $\bb{U}$ be the $\kappa$-directed closed dense subset of $\bb{S} * \dot{\bb{R}} * \dot{\bb{T}}$. 
  Define a decreasing sequence $\langle (s_n, \dot{r}_n, \dot{t}_n) \mid n < \omega \rangle$ 
  of conditions from $\bb{U}$ and an increasing sequence of ordinals $\langle \alpha_n \mid n < \omega \rangle$ 
  below $\kappa$ satisfying the following requirements.
  \begin{enumerate}
    \item $(s_0, \dot{r}_0) \leq (s, \dot{r})$.
    \item For all $n < \omega$, $(s_{n+1}, \dot{r}_{n+1})$ decides the value of $\dot{A} \cap \gamma^{s_n}$.
    \item For all $n < \omega$, $\gamma^{s_n} < \alpha_n < \gamma^{s_{n+1}}$ and 
      $(s_{n+1}, \dot{r}_{n+1}) \Vdash ``\alpha_n \in \dot{C}."$
  \end{enumerate}
  The construction is straightforward. Let $\beta = \sup(\{\alpha_n \mid n < \omega\}) = 
  \sup(\{\gamma^{s_n} \mid n < \omega\})$, and use the closure of $\bb{U}$ to find a lower 
  bound $(s_\omega, \dot{r}_\omega)$ for $\langle (s_n, \dot{r}_n) \mid n < \omega \rangle$ 
  with $\gamma^{s_\omega} = \beta$. Then $(s_\omega, \dot{r}_\omega)$ decides the value of 
  $\dot{A} \cap \beta$ to be some $B \subseteq \beta$ and forces $\beta$ to be in $\dot{C}$. 
  It is now easy to build a condition $s^* \leq s_\omega$ such that $\gamma^{s^*} = \beta \cdot 2$ 
  and $D^{s^*}_{\beta \cdot 2}$ is the ordinal closure of $\{\beta + \alpha + 1 \mid \alpha \in B\}$. 
  Let $\dot{r}^* = \dot{r}_\omega$. Then $(s^*, \dot{r}^*) \leq (s, r)$ and $(s^*, \dot{r}^*) \Vdash_{\bb{S} * \dot{\bb{R}}} ``\beta 
  \in \dot{C}$ and $\dot{A} \cap \beta = \dot{A}_\beta."$ 
\end{proof}

We can also get the consistency of certain instances of $\boxtimes^-(\kappa, \lambda)$ with 
some amount of simultaneous stationary reflection. In \cite{hayut_lh}, Hayut and the author 
prove the consistency, from large cardinals, of an indexed version of $\square(\kappa, \lambda)$ 
together with simultaneous stationary reflection. In particular, in all of these models 
in which $\kappa$ is either a successor of a singular cardinal or inaccessible, an 
examination of the proofs in \cite{hayut_lh}, together with
the arguments of Lemma \ref{reflection_lem}, implies that, for all cardinals $\mu \leq \nu < \kappa$ 
with $\nu$ regular, if $S \subseteq S^\kappa_{<\mu}$ is stationary, then $S$ reflects at an ordinal 
in $S^\kappa_\nu$. A straightforward modification of the proofs 
of Lemma \ref{higher_width_lemma} (where the length of the recursive construction in the proof is 
equal to $\lambda$, the width of the square sequence) and Lemma \ref{diamond_lem} then shows that, in the models from 
\cite{hayut_lh}, the indexed $\square(\kappa, \lambda)$-sequence is in fact a 
$\boxtimes^-(\{S^\kappa_{\nu} \mid \nu$ regular, $\lambda \leq \nu < \kappa\}, \lambda)$-sequence. In all cases, $\mathrm{GCH}$ 
can easily be arranged in the final model, so we can thus obtain, for example, the following results (the reader 
is referred to \cite{hayut_lh} for details).
\begin{corollary}
  Suppose there is a weakly compact cardinal. For $i < 2$, there is a forcing extension 
  in which $\mathrm{GCH}$, $\boxtimes^-(\{S^{\aleph_2}_{\aleph_k} \mid i \leq k < 2\}, \aleph_i)$, and $\mathrm{Refl}(<\aleph_i, S^{\aleph_2}_{\aleph_0})$ 
  hold.
\end{corollary}

\begin{corollary}
  Suppose there are infinitely many supercompact cardinals, and let $m < \omega$. There is a forcing 
  extension in which $\mathrm{GCH}$ and $\boxtimes^-(\{S^{\aleph_{\omega + 1}}_{\aleph_n} \mid m \leq n < \omega\}, \aleph_m)$ hold and 
  $\mathrm{Refl}(< \aleph_m, S^{\aleph_{\omega + 1}}_{< \aleph_n})$ holds for all $n < \omega$.
\end{corollary}

\begin{corollary}
  Suppose $\kappa$ is an inaccessible limit of supercompact cardinals, with $\langle \lambda_i 
  \mid i < \kappa \rangle$ an increasing sequence of supercompact cardinals below $\kappa$. 
  For $\ell < \kappa$, there is a forcing extension in which $\kappa$ is the least inacessible cardinal, 
  $\lambda_i$ remains a regular cardinal for all $i < \kappa$, and $\mathrm{GCH}$, $\diamondsuit(\kappa)$, 
  $\boxtimes^-(\{S^\kappa_\delta \mid \delta$ regular, $\lambda_\ell \leq \delta < \kappa \}, \lambda_\ell)$, and 
  $\mathrm{Refl}(<\lambda_\ell, \kappa)$ hold.
\end{corollary}

In light of item (1) from Theorem \ref{brodsky_rinot_thm}, these corollaries show that, for regular, uncountable $\kappa > \omega_1$, 
the existence of a $\kappa$-Souslin tree is compatible with a high degree of simultaneous stationary reflection at $\kappa$. In addition, 
in \cite{brodsky_rinot_two}, Brodsky and Rinot show that, if $\kappa$ is $<\delta$-inaccessible, i.e. $\lambda^\xi < \kappa$ for all 
$\lambda < \kappa$ and $\xi < \delta$, then $\boxtimes^-(S^\kappa_{\geq \delta}, < \kappa) + \diamondsuit(\kappa)$ implies the existence of a 
$\delta$-complete $\kappa$-Souslin tree. In the presence of appropriate cardinal arithmetic, then, our corollaries can yield highly 
complete Souslin trees.
\section{Stronger guessing principles}

In this section, we look at the relationship between $\boxtimes(\kappa)$ and stationary 
reflection. We start by proving Theorem \ref{incon_thm}. We actually prove the following 
slightly stronger result.

\begin{theorem} \label{incon_thm_2}
  Suppose $\kappa$ is a regular, uncountable cardinal, $\mathrm{Refl}(S^\kappa_\omega)$ holds, and $\vec{C} = 
  \langle C_\alpha \mid \alpha < \kappa \rangle$ is a $\square(\kappa)$-sequence. 
  Then there is a club $E \subseteq \kappa$ such that, for every 
  $\beta \in \kappa$, $\{\alpha < \beta \mid \mathrm{succ}_\omega(C_\beta \setminus \alpha) 
  \subseteq E\} = \emptyset$. In particular, $\boxtimes(\kappa)$ fails.
\end{theorem}

\begin{proof}
  Let $T = \{\alpha \in S^\kappa_\omega \mid \acc(C_\alpha)$ is bounded below $\alpha\}$.

  \begin{claim}
    $T$ is non-stationary.
  \end{claim}

  \begin{proof}
    Fix $\beta \in S^\kappa_{>\omega}$, let $D_0 = \acc(C_\beta)$, and let 
    $D_1 = \acc(D_0)$. $D_1$ is club in $\beta$ and, for every $\alpha \in D_1$, we have 
    $C_\beta \cap \alpha = C_\alpha$ and $\sup(\acc(C_\beta \cap \alpha)) = \alpha$. 
    In particular, $D_1 \cap T = \emptyset$. Therefore, $T$ does not reflect and, 
    since $\mathrm{Refl}(S^\kappa_\omega)$ holds, $T$ is non-stationary.
  \end{proof}

  Thus, there is a club $E \subseteq \kappa$ such that $E \cap T = \emptyset$. Let 
  $\beta < \kappa$, and suppose that there is $\alpha < \beta$ such that 
  $\mathrm{succ}_\omega(C_\beta \setminus \alpha) \subseteq E$. Let 
  $\delta = \sup(\mathrm{succ}_\omega(C_\beta \setminus \alpha))$. Then 
  $C_\beta \cap \delta = C_\delta$, so $\max(\acc(C_\delta)) \leq \alpha < \delta$, 
  and therefore $\delta \in T$. But $\delta$ is also a limit point of $E$ and hence 
  in $E$, contradicting the fact that $E \cap T = \emptyset$. Thus, $E$ is as specified in the 
  statement of the theorem.
\end{proof}

However, this is in some sense the only obstacle to the consistency of 
$\mathrm{GCH} + \boxtimes(\kappa) + \mathrm{Refl}(\kappa)$. For example, 
we can prove the following theorem.

\begin{theorem}
  Suppose there are infinitely many supercompact cardinals. Then there is a forcing 
  extension in which $\mathrm{GCH}$ holds and there is a non-reflecting stationary 
  set $T \subseteq S^{\aleph_{\omega + 1}}_\omega$ such that, letting $S := \aleph_{\omega + 1} \setminus T$, 
  we have both $\boxtimes(\mathrm{NS}^+_{\aleph_{\omega+1}} \restriction S)$ and $\mathrm{Refl}(S)$.
\end{theorem}

\begin{proof}
  The proof is similar to those in Section \ref{con_sect}, so we will omit some 
  repeated arguments. 
  Let $V$ be the model from the beginning of Section \ref{con_sect} obtained by 
  collapsing infinitely many supercompact cardinals to be the $\aleph_n$'s for 
  $0 < n < \omega$. Let $\mu = \aleph_\omega$ and $\kappa = \aleph_{\omega + 1}$. 
  By standard arguments, $\mathrm{Refl}(\kappa)$ 
  holds in any forcing extension of $V$ by a $\kappa$-directed closed 
  forcing poset.

  Let $\bb{S} = \bb{S}(\kappa)$. Let $G$ be $\bb{S}$-generic over $V$. In $V[G]$, 
  let $\vec{C} = \langle C_\alpha \mid \alpha < \kappa 
  \rangle$ be the generically-added $\square(\kappa)$-sequence, and let $\bb{T} = 
  \bb{T}_1(\vec{C})$. Let $T = \{\alpha < \kappa \mid \acc(C_\alpha)$ is bounded 
  below $\alpha\}$. 
  By an easy genericity argument, $T$ is a stationary subset of $S^\kappa_\omega$ 
  and, by the arguments in the proof of Theorem \ref{incon_thm_2}, $T$ does not 
  reflect. In $V[G]$ and any forcing extension thereof, we say a stationary subset 
  $S \subseteq \kappa \setminus T$ is \emph{fragile} if $\Vdash_{\bb{T}}``S$ is 
  non-stationary$."$ In $V$, let $\dot{T}$ be a canonical $\bb{S}$-name for $T$.
  
  As in Section \ref{con_sect}, define in 
  $V[G]$ a poset $\bb{R}$ so that, in the extension by $\bb{R}$, any fragile subset of 
  $\kappa \setminus T$ is non-stationary. Let $\langle \bb{R}_\eta \mid \eta \leq \kappa^+ 
  \rangle$, $\langle \dot{S}_\eta \mid \eta < \kappa^+ \rangle$, and 
  $\langle \dot{E}_\eta \mid \eta < \kappa^+ \rangle$ be as defined in Section \ref{con_sect}, 
  so that, for all $\eta < \kappa^+$, $\dot{S}_\eta$ is an $\bb{R}_\eta$-name for a 
  fragile subset of $\kappa \setminus T$ and $\dot{E}_\eta$ is an $\bb{R}_\eta \times \bb{T}$-name 
  for a club in $\kappa$ disjoint from $\dot{S}_\eta$. As in Section \ref{con_sect}, we have 
  that, in $V$, $\bb{S} * \dot{\bb{R}} * \dot{\bb{T}}$ has a dense $\kappa$-directed closed subset, 
  $\bb{U}$, consisting of all $(s, \dot{r}, \dot{t})$ such that:
  \begin{itemize}
    \item there are $r$, $t \in V$ such that $s \Vdash ``\dot{r} = \check{r}$ and 
      $\dot{t} = \check{t}"$;
    \item $\gamma^s = \sup(t)$;
    \item for all $\xi \in \dom(r)$, $(p, \dot{r} \restriction \xi, \dot{t}) \Vdash ``\max(r(\xi)) 
      \in \dot{E}_\eta$.
  \end{itemize}

  Let $H$ be $\bb{R}$-generic over $V[G]$. $V[G*H]$ will be our desired model. 
  Let $S = \kappa \setminus T$.
  $\mathrm{GCH}$ easily holds in $V[G*H]$, and arguments exactly as in Section \ref{con_sect} 
  show that $\mathrm{Refl}(S)$ holds. To finish the proof, 
  it will thus suffice to show that $\vec{C}$ is a $\boxtimes(\mathrm{NS}^+_\kappa \restriction S)$-sequence. 
  To this end, the following lemma will be useful.

  \begin{lemma} \label{extension_lemma}
    In $V[G*H]$, suppose $\gamma < \kappa$ is a limit ordinal and $A$ is an unbounded 
    subset of $\kappa$. Then there is a limit ordinal $\delta < \kappa$ such that 
    $\gamma \in \acc(C_\delta)$ and $\mathrm{succ}_\omega(C_\delta \setminus (\gamma + 1)) 
    \subseteq A$.
  \end{lemma}

  \begin{proof}
    Work in $V$. Let $\dot{A}$ be an $\bb{S} * \dot{\bb{R}}$-name for $A$, 
    and let $(s, \dot{r}) \in \bb{S} * \dot{\bb{R}}$. We will 
    find $(s^*, \dot{r}^*) \leq (s, \dot{r})$ forcing the conclusion 
    of the lemma.

    Without loss of generality, assume that $\gamma^s \geq \gamma$. Recursively construct 
    a decreasing sequence $\langle (s_n, \dot{r}_n, \dot{t}_n) \mid 
    n < \omega \rangle$ from $\bb{U}$ together with an increasing sequence of ordinals 
    below $\kappa$, $\langle \eta_n \mid n < \omega \rangle$, satisfying the following 
    conditions. 
    \begin{enumerate}
      \item $(s_0, \dot{r}_0) \leq (s, \dot{r})$;
      \item for all $n < \omega$, letting $r_n$ and $t_n$ be the elements of $V$ 
        specified in the definition of $\bb{U}$, we have that $\sup(\{\gamma^{s_n}\} 
        \cup \{\max(r_n(\xi)) \mid \xi \in \dom(r_n)\}) < \eta_n < \min(\{\gamma^{s_{n+1}}\}
        \cup \{\max(r_{n+1}(\xi)) \mid \xi \in \dom(r_{n+1})\})$;
      \item for all $n < \omega$, $(s_{n+1}, \dot{r}_{n+1}) \Vdash ``\eta_n \in \dot{A}."$ 
    \end{enumerate}

    The construction is straightforward. Let $\delta = \sup(\{\eta_n \mid n < \omega\})$, and define 
    $(s^*, \dot{r}^*)$ as follows. First, let $s^* = \langle C^{s^*}_\alpha \mid \alpha \leq \delta \rangle$, 
    where, for $\alpha < \delta$, $C^{s^*}_\alpha = C^{s_n}_\alpha$ for some $n < \omega$ such that 
    $\alpha \leq \gamma^{s_n}$ and $C^{s^*}_\delta = C^{s^*}_\gamma \cup \{\gamma\} \cup \{\eta_n \mid n < \omega\}$. 
    Then $s^* \in \bb{S}$ and $s^* \leq s_n$ for all $n < \omega$.

    Let $X = \bigcup_{n < \omega} \dom(r_n)$, and let $r^*$ be a function such that $\dom(r^*) = X$ and, 
    for all $\xi \in X$, $r^*(\xi) = \{\delta\} \cup \bigcup \{r_n(\xi) \mid \xi \in \dom(r_n)\}$. Let $\dot{r}^*$ 
    be an $\bb{S}$-name such that $(p^*) \Vdash ``\dot{r}^* = r^*."$ It is 
    straightforward to verify, by induction on $\xi \in X \cup \{\kappa^+\}$, that $(s^*, \dot{r}^* \restriction \xi) \in 
    \bb{S} * \dot{\bb{R}}_\xi$ and $(s^*, \dot{r}^* \restriction \xi) \leq (s_n, 
    \dot{r}_n \restriction \xi)$ for all $n < \omega$. This verification uses the fact that, 
    since $\max(\acc(C^{s^*}_\delta)) = \gamma < \delta$, we have $s^* \Vdash ``\delta \in \dot{T},"$ and, 
    therefore, for all $\xi \in \dom(r^*)$, $(s^*, \dot{r}^* \restriction \xi) \Vdash ``\delta \not\in \dot{S}_\xi."$
    But now $(s^*, \dot{r}^*) \leq (s, \dot{r})$ and forces 
    that $\delta$ is as desired in the statement of the lemma.
  \end{proof}

  Towards a contradiction, suppose that, in $V[G*H]$, $S_0 \subseteq S$ is stationary, 
  $\vec{A} = \langle A_i \mid i < \kappa \rangle$ is a sequence of unbounded subsets of $\kappa$, 
  and $S_0$ and $\vec{A}$ witness that $\vec{C}$ is not a 
  $\boxtimes(\mathrm{NS}^+_\kappa \restriction S)$-sequence. By two applications of Fodor's Lemma, we can find a 
  stationary $S' \subseteq S_0$, and fixed $i^*, \zeta^* < \kappa$ such that, 
  for all $\beta \in S'$, $\sup(\{\alpha < \beta \mid \mathrm{succ}_\omega(C_\beta \setminus 
  \alpha) \subseteq A_{i^*}) = \zeta^*$. Since $S'$ is not fragile, there is $t_0 \in \bb{T}$ 
  such that $t_0 \Vdash_{\bb{T}}``S'$ is stationary$."$ Let $\gamma < \kappa$ 
  be such that $t = C_\gamma$. Without loss of generality, $\gamma > \zeta^*$. 
  Use Lemma \ref{extension_lemma} to find a limit ordinal $\delta < \kappa$ 
  such that $\gamma \in \acc(C_\delta)$ and $\mathrm{succ}_\omega(C_\delta \setminus (\gamma + 1)) 
  \subseteq A_{i^*}$. Let $t^* = C_\delta$. Then $t^* \leq t$. Let $I$ be $\bb{T}$-generic 
  over $V[G*H]$ with $t^* \in I$. Let $D = \bigcup I$. $D$ is a club in $\kappa$ and, 
  for all $\beta \in \acc(D) \setminus \delta$, $\sup(\{\alpha < \beta \mid \mathrm{succ}_\omega(C_\beta 
  \setminus \alpha) \subseteq A_{i^*}) \geq \gamma + 1 > \zeta^*$. In particular, $(\acc(D) \setminus \delta) \cap 
  S' = \emptyset$, so $S'$ is non-stationary in $V[G*H*I]$, contradicting the fact that 
  $t^* \in I$, $t^* \leq t$, and $t \Vdash ``S'$ is stationary$."$
\end{proof}

\section{Trees with narrow ascent paths}

In this section, we prove Theorem \ref{asc_path_thm}. We first show that $\square^{\mathrm{ind}}_{\mu, \cf{\mu}}$ 
suffices to obtain a special $\mu^+$-tree with a narrow ascent path.

\begin{theorem}
  Suppose $\mu$ is a singular cardinal and $\square^{\mathrm{ind}}_{\mu, cf(\mu)}$ holds. 
  Then there is a special $\mu^+$-tree with a $\cf(\mu)$-ascent path.
\end{theorem}

\begin{proof}
  Let $\langle C_{\alpha, i} \mid \alpha < \mu^+, i(\alpha) \leq i < \cf(\mu) \rangle$ be a 
  $\square^{\mathrm{ind}}_{\mu, \cf(\mu)}$-sequence. If $\beta = \alpha + 1 < \mu^+$ is a 
  successor ordinal, we may assume that $i(\beta) = 0$ and $C_{\beta, i} = \{\alpha\}$ for 
  all $i < \cf(\mu)$. We will define a special 
  $\mu^+$-tree using Todorcevic's method of minimal walks.

  Suppose $\alpha < \beta < \mu^+$ and $i(\beta) \leq i < \cf(\mu)$. Then the $i^{\mathrm{th}}$ 
  minimal walk from $\beta$ to $\alpha$ is a finite sequence of pairs 
  $\langle (\beta_m, i_m) \mid m < n_i(\alpha, \beta) \rangle$, decreasing in the first coordinate, satisfying $\alpha < \beta_m \leq \beta$ 
  and $i(\beta_m) \leq i_m < \cf(\mu)$ for all $m < n_i(\alpha, \beta)$, constructed by recursion on $m$ 
  as follows:
  \begin{itemize}
    \item $\beta_0 = \beta$ and $i_0 = i$;
    \item if $(\beta_m, i_m)$ is defined and $\alpha \in C_{\beta_m, i_m}$, then let 
      $n_i(\alpha, \beta) = m+1$ and stop the construction;
    \item if $(\beta_m, i_m)$ is defined and $\alpha \not\in C_{\beta_m, i_m}$, then let 
      $\beta_{m+1} = \min(C_{\beta_m, i_m} \setminus \alpha)$ and let $i_{m+1} = i(\beta_{m+1})$.
  \end{itemize}
  Define the \emph{projection} of this walk, $\mathrm{pr}_i(\alpha, \beta)$, to be the sequence 
  $\langle C_{\beta_m, i_m} \cap \alpha \mid m < n_i(\alpha, \beta) \rangle$, and define the 
  \emph{trace} of the walk, $\mathrm{tr}_i(\alpha, \beta)$, to be the sequence 
  $\langle \otp(C_{\beta_m, i_m} \cap \alpha) \mid m < n_i(\alpha, \beta) \rangle$. For convenience, 
  define $\mathrm{pr}_i(\alpha, \alpha) = \emptyset = \mathrm{tr}_i(\alpha, \alpha)$ for all 
  $\alpha < \mu^+$ and $i < \cf(\mu)$.

  We now prove some basic facts about these sequences. Recall first the Kleene-Brouwer ordering $<_{\mathrm{KB}}$ on 
  ${^{<\omega}} \mu$, the set of finite sequences from $\mu$. If $\sigma$ and $\tau$ are distinct elements of 
  ${^{<\omega}} \mu$, then $\sigma <_{\mathrm{KB}} \tau$ iff either $\sigma$ end-extends $\tau$ or $\sigma(m) < \tau(m)$, 
  where $m$ is the least natural number such that $\sigma(m) \neq \tau(m)$.

  \begin{lemma} \label{walk_lemma_1}
    Suppose $\alpha_0 < \alpha_1 < \beta < \mu^+$ and $i(\beta) \leq i < \cf(\mu)$. 
    \begin{enumerate}
      \item $\mathrm{tr}_i(\alpha_0, \beta) <_{\mathrm{KB}} \mathrm{tr}_i(\alpha_1, \beta)$.
      \item If $\alpha_1$ and $\beta$ are limit ordinals and $\alpha_1 \in \acc(C_{\beta, i})$, 
        then $\mathrm{pr}_i(\alpha_0, \alpha_1) = \mathrm{pr}_i(\alpha_0, \beta)$.
    \end{enumerate}
  \end{lemma}
  
  \begin{proof}
    We show (1). For $k < 2$, let $\langle (\beta^k_m, i^k_m) \mid m < n_i(\alpha_k, \beta) \rangle$ be the 
    $i^{\mathrm{th}}$ minimal walk from $\beta$ to $\alpha_k$.

    \textbf{Case 1: For all $m < n_i(\alpha_1, \beta)$, $C_{\beta^1_m, i^1_m} \cap [\alpha_0, \alpha_1) = \emptyset$.} 
      In this case, $(\beta^0_m, i^0_m) = (\beta^1_m, i^1_m)$ for all $m < n_i(\alpha_1, \beta)$, $\beta^0_{n_i(\alpha_1, \beta)} = \alpha_1$, 
      and therefore $\mathrm{tr}_i(\alpha_0, \beta)$ end-extends $\mathrm{tr}_i(\alpha_1, \beta)$.

    \textbf{Case 2: Otherwise.} Let $m$ be least such that $C_{\beta^1_m, i^1_m} \cap [\alpha_0, \alpha_1) \neq \emptyset$. 
      Then $(\beta^0_\ell, i^0_\ell) = (\beta^1_\ell, i^1_\ell)$ for all $\ell \leq m$, 
      $C_{\beta^0_\ell, i^0_\ell} \cap \alpha_1 = C_{\beta^0_\ell, i^0_\ell} \cap \alpha_0$ for all $\ell < m$, 
      and $\otp(C_{\beta^0_m, i^0_m} \cap \alpha_0) < \otp(C_{\beta^0_m, i^0_m} \cap \alpha_1)$, so, again, 
      $\mathrm{tr}_i(\alpha_0, \beta) <_{\mathrm{KB}} \mathrm{tr}_i(\alpha_1, \beta)$.

    (2) follows directly from the definitions.
  \end{proof}

  \begin{lemma} \label{walk_lemma_2}
    Suppose $\alpha_0 < \alpha_1 < \beta_0 < \beta_1 < \mu^+$ and $i(\beta_0), i(\beta_1) \leq i < \cf(\mu)$.
    Suppose also that $\mathrm{pr}_i(\alpha_1, \beta_0) = 
    \mathrm{pr}_i(\alpha_1, \beta_1)$. Then $\mathrm{pr}_i(\alpha_0, \beta_0) = \mathrm{pr}_i(\alpha_0, \beta_1)$.
  \end{lemma}

  \begin{proof}
    Let $n = n_i(\alpha_1, \beta_0) = n_i(\alpha_1, \beta_1)$. For $k < 2$, let $\langle (\beta^k_m, i^k_m) \mid m < n \rangle$ be the $i^{\mathrm{th}}$ 
    minimal walk from $\beta_k$ to $\alpha_1$. Let $\langle D_m \mid m < n \rangle = \mathrm{pr}_i(\alpha_1, \beta_0) = \mathrm{pr}_i(\alpha_1, \beta_1)$. 
    Suppose first that, for all $m < n$, $D_m \cap [\alpha_0, \alpha_1) = \emptyset$. In this case, we have
    $\mathrm{pr}_i(\alpha_0, \beta_0) = \mathrm{pr}_i(\alpha_1, \beta_0) {^\frown} \mathrm{pr}_{i(\alpha_1)}(\alpha_0, \alpha_1) = 
    \mathrm{pr}_i(\alpha_0, \beta_1)$, and we are done. Otherwise, let $m < n$ be least such that $D_m \cap [\alpha_0, \alpha_1) \neq \emptyset$. 
    Let $\gamma = \min(D_m \setminus \alpha_0)$. Then 
    $\mathrm{pr}_i(\alpha_0, \beta_0) = \mathrm{pr}_i(\alpha_1, \beta_0) \restriction m {^\frown} \langle D_m \cap \alpha_0 \rangle 
    {^\frown} \mathrm{pr}_{i(\gamma)}(\alpha_0, \gamma) = \mathrm{pr}_i(\alpha_0, \beta_1)$.
  \end{proof}

  \begin{lemma} \label{size_lemma}
    Fix $\alpha < \mu^+$, and let $\mathcal{D}_\alpha = \{C_{\beta, i} \cap \alpha \mid \beta < \mu^+, i(\beta) \leq i < \cf(\mu) \}$. 
    Then $|\mathcal{D}_\alpha| \leq \mu$.
  \end{lemma}

  \begin{proof}
    Suppose $D \in \mathcal{D}_\alpha$. If $\alpha$ is limit and $\sup(D) = \alpha$, then $D = C_{\alpha, i}$ for some 
    $i(\alpha) \leq i < \cf(\mu)$. Otherwise, let $\gamma = \max(\acc(D))$. Then $D$ is the union of $C_{\gamma, i}$ for 
    some $i(\gamma) \leq i < \cf(\mu)$ and a finite subset of $\alpha$. In either case, there are at most $\mu$ choices for $D$.
  \end{proof}

  We are now ready to define our desired $\mu^+$-tree, $T$. Elements of $T$ will be all sequences of the form 
  $\langle \mathrm{tr}_i(\alpha, \beta) \mid \alpha \leq \alpha^* \rangle$, where $\alpha^* < \beta < \mu^+$ and 
  $i(\beta) \leq i < \cf(\mu)$. If $s, t \in T$, then $s <_T t$ iff $t$ end-extends $s$. $T$ is thus manifestly a 
  tree of height $\mu^+$ and, for all $\alpha^* < \mu^+$, $T_{\alpha^*}$ is precisely the set of 
  sequences in $T$ of length $\alpha^* + 1$.

  \begin{claim}
    $T$ is a $\mu^+$-tree, i.e. $|T_{\alpha^*}| \leq \mu$ for all $\alpha^* < \mu^+$.
  \end{claim}

  \begin{proof}
    Fix $\alpha^* < \mu^+$. By Lemma \ref{walk_lemma_2}, for all $\alpha^* < \beta < \mu^+$ and all 
    $i(\beta) \leq i < \mu^+$, 
    $\langle \mathrm{tr}_i(\alpha, \beta) \mid \alpha \leq \alpha^* \rangle$ is determined 
    by $\mathrm{pr}_i(\alpha^*, \beta)$.
    Since $\mathrm{pr}_i(\alpha^*, \beta)$ is a finite sequence from $\mathcal{D}_{\alpha^*}$, 
    Lemma \ref{size_lemma} implies that $|T_{\alpha^*}| \leq \mu$.
  \end{proof}

  \begin{claim}
    $T$ is special.
  \end{claim}

  \begin{proof}
    Let $f$ be a bijection from $\mu^{<\omega}$ to $\mu$. By (1) of Lemma \ref{walk_lemma_1}, 
    each element of $T$ is a $<_{\mathrm{KB}}$-increasing sequence from $\mu^{<\omega}$. In 
    particular, each element of $T$ is an injective sequence, so the function 
    $t \mapsto f(t(\max(\dom(t)))$ witnesses that $T$ is special.
  \end{proof}

  \begin{claim}
    $T$ has a $\cf(\mu)$-ascent path.
  \end{claim}

  \begin{proof}
    For $\alpha^* < \mu^+$, let $\beta_{\alpha^*} = \alpha^* + \omega$. Define $\langle 
    b_{\alpha^*}:\cf(\mu) \rightarrow T_{\alpha^*} \mid \alpha^* < \mu^+ \rangle$ as follows. 
    If $\alpha^* < \mu^+$ and $i < i(\beta_{\alpha^*})$, let $b_{\alpha^*}(i)$ be an 
    arbitrary element of $T_{\alpha^*}$. If $i(\beta_{\alpha^*}) \leq i < \cf(\mu)$, 
    then let $b_{\alpha^*}(i) = \langle \mathrm{tr}_i(\alpha, \beta_{\alpha^*}) \mid 
    \alpha \leq \alpha^* \rangle$.

    Fix $\alpha_0 < \alpha_1 < \mu^+$. Let $i^* < \cf(\mu)$ be least such that 
    $\beta_{\alpha_0} \in \acc(C_{\beta_{\alpha_1}})$ (if $\beta_{\alpha_0} = \beta_{\alpha_1}$, 
    $i^*$ is simply $i(\beta_{\alpha_0})$). By (2) of Lemma \ref{walk_lemma_1}, for all 
    $\alpha \leq \alpha_0$ and all $i^* \leq i < \cf(\mu)$, $\mathrm{tr}_i(\alpha, \beta_{\alpha_0}) = 
    \mathrm{tr}_i(\alpha, \beta_{\alpha_1})$. Thus, for all $i^* \leq i < \cf(\mu)$, 
    $b_{\alpha_0}(i) <_T b_{\alpha_1}(i)$, so $\langle b_{\alpha^*} \mid \alpha^* < \mu^+ \rangle$ 
    is a $\cf(\mu)$-ascent path.
  \end{proof}
  \let\qed\relax
\end{proof}

Clause (1) of Theorem \ref{asc_path_thm} now follows from the following result.

\begin{theorem}
  Suppose $\mu$ is a singular cardinal and $\square_\mu$ holds. Then $\square^{\mathrm{ind}}_{\mu, \cf(\mu)}$ holds.
\end{theorem}

\begin{proof}
  Suppose $\vec{D} = \langle D_\alpha \mid \alpha < \mu^+ \rangle$ is a $\square_\mu$-sequence. Since $\mu$ is singular, 
  we may assume that $\otp(D_\alpha) < \mu$ for all $\alpha < \mu^+$. Let $\langle \mu_i \mid i < \cf(\mu) \rangle$ be 
  an increasing sequence of regular, uncountable cardinals, cofinal in $\mu$. We will define a $\square^{\mathrm{ind}}_{\mu, \cf(\mu)}$-sequence 
  $\vec{C} = \langle C_{\alpha, i} \mid \alpha < \mu^+, i(\alpha) \leq i < \cf(\mu) \rangle$ such that, for all $\alpha < \mu^+$ 
  and all $i(\alpha) \leq i < \cf(\mu)$, $\otp(C_{\alpha, i}) < \mu_i$. The construction is by induction on $\alpha$, 
  maintaining the following inductive hypotheses:
  \begin{itemize}
    \item for all $\alpha < \mu^+$, $i(\alpha)$ is the least $i < \cf(\mu)$ such that $\otp(D_\alpha) < \mu_i$;
    \item for all $\alpha < \mu^+$, $\acc(D_\alpha) \subseteq \acc(C_{\alpha, i(\alpha)})$.
  \end{itemize}
  Thus, suppose $\beta < \mu^+$ is a limit ordinal and we have defined $\langle C_{\alpha, i} \mid \alpha < \beta, 
  i(\alpha) \leq i < \cf(\mu) \rangle$. $i(\beta)$ is the least $i < \cf(\mu)$ such that $\otp(D_\beta) < \mu_i$. 
  We now consider three separate cases.

  \textbf{Case 1: $\sup(\acc(D_\beta)) = \beta$.} In this case, note that, for all $\alpha \in \acc(D_\beta)$, we have 
  $\otp(D_\alpha) < \otp(D_\beta) < \mu_{i(\beta)}$, so we have $i(\alpha) \leq i(\beta)$. For all $i(\beta) \leq i < \cf(\mu)$, 
  let $C_{\beta,i} = \bigcup_{\alpha \in \acc(D_\beta)} C_{\alpha, i}$. To see that $\otp(C_{\beta, i}) < \mu_i$, note that, 
  by the inductive hypothesis, we have that $\otp(C_{\alpha, i}) < \mu_i$ for all $\alpha \in \acc(D_\beta)$. Also by the 
  inductive hypothesis, if $\alpha_0 < \alpha_1$ are both in $\acc(D_\beta)$, then, since $\alpha_0 \in \acc(D_{\alpha_1})$, 
  we have $C_{\alpha_1, i} \cap \alpha_0 = C_{\alpha_0, i}$ for all $i(\beta) \leq i < \cf(\mu)$. Therefore, every 
  initial segment of $C_{\beta, i}$ has order type $< \mu_i$. Since $\acc(D_\beta)$ is a cofinal subset of $C_{\beta, i}$ and 
  $\otp(D_\beta) < \mu_i$, this implies that $\otp(C_{\beta, i}) < \mu_i$. All of the other requirements in the definition of 
  $\square^{\mathrm{ind}}_{\mu, \cf(\mu)}$ and the inductive hypotheses are straightforward to verify.

  \textbf{Case 2: $\beta$ is a limit of limit ordinals and $\sup(\acc(D_\beta)) < \beta$.} In this case, 
  $\cf(\beta) = \omega$. Let $\langle \alpha_n \mid n < \omega \rangle$ be an increasing sequence of limit ordinals, 
  cofinal in $\beta$, such that, if $\otp(D_\beta) \neq \omega$, then $\alpha_0 = \sup(\acc(D_\beta))$. Let 
  $\langle i_n \mid n < \omega \rangle$ be a strictly increasing sequence of ordinals below $\cf(\mu)$ such that 
  $i_0 = i(\beta)$ and, for $0 < n < \omega$, we have $\{\alpha_m \mid m < n\} \subseteq \acc(C_{\alpha_n, i_n})$. 
  Fix $i$ such that $i(\beta) \leq i < \cf(\mu)$, and let us define $C_{\beta, i}$. If there is $n < \omega$ 
  such that $i \in [i_n, i_{n+1})$, then let $C_{\beta, i} = C_{\alpha_n, i} \cup \{\alpha_\ell \mid n \leq \ell < \omega\}$. 
  If $i \geq \sup(\{i_n \mid n < \omega\})$, then let $C_{\beta, i} = \bigcup_{n < \omega} C_{\alpha_n, i}$. It is 
  straightforward to verify that this satisfies our requirements.

  \textbf{Case 3: $\beta = \alpha + \omega$ for some limit ordinal $\alpha$.} Let $\alpha_0 = \sup(\acc(D_\beta))$ 
  (if $\otp(D_\beta) = \omega$, let $\alpha_0 = \omega$). Let $i^* < \cf(\mu)$ be the least $i$ such that $i(\beta) \leq i$ 
  and $\alpha_0 \in \acc(C_{\alpha, i})$ (if $\alpha_0 = \alpha$, let $i^* = i(\beta)$). For $i(\beta) \leq i < i^*$, 
  let $C_{\beta, i} = C_{\alpha_0, i} \cup \{\alpha_0\} \cup \{\alpha + n \mid n < \omega\}$. For $i^* \leq i < \cf(\mu)$, 
  let $C_{\beta, i} = C_{\alpha, i} \cup \{\alpha + n \mid n < \omega\}$. It is again straightforward to verify that 
  this satisfies all of our requirements and thus completes the construction of our $\square^{\mathrm{ind}}_{\mu, \cf(\mu)}$-sequence.
\end{proof}

For our proof of clause (2) of Theorem \ref{asc_path_thm}, the following definitions will be useful.

\begin{definition}
  Suppose $\lambda < \kappa$ are regular cardinals and $c:[\kappa]^2 \rightarrow \lambda$.
  \begin{enumerate}
    \item $c$ is \emph{subadditive} if, for all $\alpha < \beta < \gamma < \kappa$:
      \begin{enumerate}
        \item $c(\alpha, \gamma) \leq \max(c(\alpha, \beta), c(\beta, \gamma))$;
        \item $c(\alpha, \beta) \leq \max(c(\alpha, \gamma), c(\beta, \gamma))$.
      \end{enumerate}
    \item $c$ is \emph{unbounded} if, for all unbounded $A \subseteq \kappa$, $c``[A]^2$ 
      is unbounded in $\lambda$.
  \end{enumerate}
\end{definition}

Suppose $\lambda < \kappa$ are regular cardinals, $T$ is a tree of height $\kappa$ with no cofinal branch, 
and $\langle b_\gamma: \lambda \rightarrow T_\gamma \mid \gamma < \kappa \rangle$ is a $\lambda$-ascent path 
through $T$. Define a function $c:[\kappa]^2 \rightarrow \lambda$ by letting, for all $\alpha < \beta < \kappa$, 
$c(\alpha, \beta)$ be the least $\eta < \lambda$ such that, for all $\eta \leq \xi < \lambda$, $b_\alpha(\xi) 
<_T b_\beta(\xi)$. Then $c$ is easily seen to be an unbounded subadditive function. Therefore, the non-existence 
of unbounded subadditive functions from $[\kappa]^2$ to $\lambda$ will suffice to prove that all trees of 
height $\kappa$ with a $\lambda$-ascent path must have a cofinal branch.

Recall next the following definition.

\begin{definition}
  Let $\lambda < \kappa$ be regular cardinals. $\mathcal{D} = \langle D(i, \beta) \mid i < \lambda, \beta < \kappa \rangle$ 
  is a \emph{$\lambda$-covering matrix for $\kappa$} if the following hold:
  \begin{enumerate}
    \item for all $\beta < \kappa$, $\bigcup_{i < \lambda} D(i, \beta) = \beta$;
    \item for all $\beta < \kappa$ and $i < j < \lambda$, $D(i, \beta) \subseteq D(j, \beta)$;
    \item for all $\beta < \gamma < \kappa$ and $i < \lambda$, there is $j < \lambda$ 
      such that $D(i, \beta) \subseteq D(j, \gamma)$.
  \end{enumerate}
  If $\mathcal{D}$ is a $\lambda$-covering matrix for $\kappa$, $\mathcal{D}$ is called \emph{locally downward coherent} 
  if, for all $X \in [\kappa]^{\leq \lambda}$, there is $\gamma_X < \kappa$ such that, for all $\beta < \kappa$ and 
  $i < \lambda$, there is $j < \lambda$ such that $D(i, \beta) \cap X \subseteq D(j, \gamma_X)$.
\end{definition}

The following covering property was introduced by Viale in his proof that the Singular Cardinals Hypothesis follows from 
the Proper Forcing Axiom (see \cite{viale}).

\begin{definition}
  Let $\lambda < \kappa$ be regular cardinals, and let $\mathcal{D}$ be a $\lambda$-covering matrix for $\kappa$. 
  $\mathrm{CP}(\mathcal{D})$ holds if there is an unbounded $A \subseteq \kappa$ such that, for all $X \in [A]^{\leq \lambda}$, 
  there are $i < \lambda$ and $\beta < \kappa$ such that $X \subseteq D(i, \beta)$. 

  $\mathrm{CP}(\kappa, \lambda)$ is the assertion that $\mathrm{CP}(\mathcal{D})$ holds whenever $\mathcal{D}$ is a 
  locally downward coherent $\lambda$-covering matrix for $\kappa$. $\mathrm{CP}^*(\kappa)$ is the assertion that 
  $\mathrm{CP}(\kappa, \theta)$ holds for every regular $\theta$ with $\theta^+ < \kappa$.
\end{definition}

\begin{lemma}
  Suppose $\lambda < \kappa$ are regular cardinals, $\mathrm{CP}(\kappa, \lambda)$ holds, and $c:[\kappa]^2 \rightarrow \lambda$ 
  is subadditive. Then $c$ is not unbounded.
\end{lemma}

\begin{proof}
  Define $\mathcal{D} = \langle D(i, \beta) \mid i < \lambda, \beta < \kappa \rangle$ by letting, for all $i < \lambda$ and 
  $\beta < \kappa$, $D(i, \beta) = \{\alpha < \beta \mid c(\alpha, \beta) \leq i\}$. We claim that $\mathcal{D}$ is a 
  locally downward coherent $\lambda$-covering matrix for $\kappa$. Items (1) and (2) in the definition of a covering matrix 
  are immediate. To show (3), fix $\beta < \gamma < \kappa$ and $i < \lambda$. Let $j =  \max(c(\beta, \gamma), i)$. 
  By subadditiviy of $c$, $D(i, \beta) \subseteq D(j, \gamma)$. To show that $\mathcal{D}$ is 
  locally downward coherent, fix $X \in [\kappa]^{\leq \lambda}$. We claim that $\gamma_X = \sup\{\alpha + 1 \mid \alpha \in X\}$ 
  is as desired in the definition. We thus fix $\beta < \kappa$ and $i < \lambda$ and find $j < \lambda$ such that 
  $D(i, \beta) \cap X \subseteq D(j, \gamma_X)$. If $\beta \leq \gamma_X$, then we are done by (3) in the definition of a 
  covering matrix. Thus, suppose $\beta > \gamma_X$. Let $j = \max(c(\gamma_X, \beta), i)$. By subadditivity of $c$, 
  if $\alpha \in D(i, \beta) \cap X$, then $c(\alpha, \gamma_X) \leq j$, so $D(i, \beta) \cap X \subseteq D(j, \gamma_X)$.

  By $\mathrm{CP}(\kappa, \lambda)$, there is an unbounded $A \subseteq \kappa$ such that, for all $X \in [A]^{\leq \lambda}$, 
  there are $i < \lambda$ and $\beta < \kappa$ such that $X \subseteq D(i, \beta)$. We claim there is $i < \lambda$ 
  such that $c``[A]^2 \subseteq i$, which will show that $c$ is not unbounded.

  Suppose this is not the case. Then, for every $i < \lambda$, there are $\alpha_i < \beta_i$ in $A$ such that 
  $c(\alpha_i, \beta_i) \geq i$. Let $X = \{\alpha_i, \beta_i \mid i < \lambda\}$. $X \in [A]^{\leq \lambda}$, so 
  there are $i^* < \lambda$ and $\beta^* < \kappa$ such that $X \subseteq D(i^*, \beta^*)$. This means that, 
  for all $i < \lambda$, we have $c(\alpha_i, \beta^*), c(\beta_i, \beta^*) < i^*$. By subadditivity, 
  $c(\alpha_i, \beta_i) < i^*$. In particular, $c(\alpha_{i^*}, \beta_{i^*}) < i^*$, which is a contradiction.
\end{proof}

In \cite{covering_2}, we prove that, assuming the consistency of large cardinals, $\mathrm{CP}^*(\mu^+)$ 
is compatible with $\square_{\mu, 2}$ for uncountable $\mu$ and, for inaccessible $\kappa$, $\mathrm{CP}^*(\kappa)$ 
is compatible with $\square(\kappa, 2)$. We therefore obtain the following corollaries, the second of which yields clause (2) 
of Theorem \ref{asc_path_thm} (again, $\aleph_{\omega+1}$ is used just for concreteness; similar results can be obtained 
at other successors of singular cardinals). In all cases, we can easily arrange for GCH to hold in the forcing extension.

\begin{corollary}
  Suppose $\mu < \kappa$ are regular, uncountable cardinals, with $\kappa$ measurable. There is a forcing extension, 
  preserving all cardinals $\leq \mu$, in which $\kappa = \mu^+$, $\square_{\mu, 2}$ holds, and, for all 
  regular $\lambda < \mu$, there is no unbounded subadditive function $c:[\kappa]^2 \rightarrow \lambda$.
\end{corollary}

\begin{corollary}
  Suppose there are infinitely many supercompact cardinals. There is a forcing extension in which $\square_{\aleph_\omega, 2}$ 
  holds and, for all $n < \omega$, there is no unbounded subadditive function $c:[\aleph_{\omega+1}]^2 \rightarrow \aleph_n$.
\end{corollary}

\begin{corollary}
  Suppose $\kappa$ is an inaccessible limit of supercompact cardinals. There is a cardinal-preserving forcing extension
  in which $\kappa$ remains inaccessible, $\square(\kappa, 2)$ holds, and, for all regular $\lambda < \kappa$, there is 
  no unbounded subadditive function $c:[\kappa]^2 \rightarrow \lambda$.
\end{corollary}

\bibliographystyle{amsplain}
\bibliography{squares_and_reflection}

\providecommand{\bysame}{\leavevmode\hbox to3em{\hrulefill}\thinspace}
\providecommand{\MR}{\relax\ifhmode\unskip\space\fi MR }
\providecommand{\MRhref}[2]{%
  \href{http://www.ams.org/mathscinet-getitem?mr=#1}{#2}
}
\providecommand{\href}[2]{#2}
\begin{thebibliography}{10}

\bibitem{brodsky_rinot}
Ari~Meir Brodsky and Assaf Rinot, \emph{A microscopic approach to
  {S}ouslin-tree constructions, part {I}}, Preprint.

\bibitem{brodsky_rinot_two}
\bysame, \emph{A microscopic approach to {S}ouslin-tree constructions, part
  {II}}, Preprint.

\bibitem{brodsky_rinot_reduced}
\bysame, \emph{Reduced powers of {S}ouslin trees}, Preprint.

\bibitem{cfm}
James Cummings, Matthew Foreman, and Menachem Magidor, \emph{Squares, scales
  and stationary reflection}, Journal of Mathematical Logic \textbf{1} (2001),
  no.~01, 35--98.

\bibitem{cummings-schimmerling}
James Cummings and Ernest Schimmerling, \emph{Indexed squares}, Israel Journal
  of Mathematics \textbf{131} (2002), no.~1, 61--99.

\bibitem{eisworth}
Todd Eisworth, \emph{Successors of singular cardinals}, Handbook of set theory,
  Springer, 2010, pp.~1229--1350.

\bibitem{hayut_lh}
Yair Hayut and Chris Lambie-Hanson, \emph{Simultaneous stationary reflection
  and square sequences}, Preprint.

\bibitem{jech}
Thomas Jech, \emph{Set theory, the third millennium, revised and expanded ed},
  Springer Monographs in Mathematics, Springer-Verlag, Berlin (2003).

\bibitem{jensen}
R.~Bj{\"o}rn Jensen, \emph{The fine structure of the constructible hierarchy},
  Ann. Math. Logic \textbf{4} (1972), 229--308; erratum, ibid. 4 (1972), 443,
  With a section by Jack Silver.

\bibitem{covering_2}
Chris Lambie-Hanson, \emph{Covering properties and square principles}, Israel
  Journal of Mathematics, To appear.

\bibitem{covering}
\bysame, \emph{Squares and covering matrices}, Annals of Pure and Applied Logic
  \textbf{165} (2014), no.~2, 673--694.

\bibitem{laver_shelah}
Richard Laver and Saharon Shelah, \emph{The $\aleph_2$-{S}ouslin hypothesis},
  Transactions of the American Mathematical Society \textbf{264}, no.~2,
  411--417.

\bibitem{luecke}
Philipp L\"{u}cke, \emph{Ascending paths and forcings that specialize higher
  {A}ronszajn trees}, Preprint.

\bibitem{magidor}
Menachem Magidor, \emph{Reflecting stationary sets}, The Journal of Symbolic
  Logic \textbf{47} (1982), no.~4, 755--771.

\bibitem{assaf}
Assaf Rinot, Private communication.

\bibitem{rinot}
\bysame, \emph{Higher {S}ouslin trees and the {GCH}, revisited}, Preprint.

\bibitem{rinot_approachability}
\bysame, \emph{A relative of the approachability ideal, diamond and
  non-saturation}, The Journal of Symbolic Logic (2010), 1035--1065.

\bibitem{sakai}
Hiroshi Sakai, \emph{Chang's {C}onjecture and weak square}, Archive for
  Mathematical Logic \textbf{52} (2013), no.~1-2, 29--45.

\bibitem{schimmerling}
Ernest Schimmerling, \emph{Combinatorial principles in the core model for one
  {W}oodin cardinal}, Ann. Pure Appl. Logic \textbf{74} (1995), no.~2,
  153--201.

\bibitem{shani}
Assaf Shani, \emph{Fresh subsets of ultrapowers}, Archive for Mathematical
  Logic, To appear.

\bibitem{shelah_stanley}
Saharon Shelah and Lee Stanley, \emph{Weakly compact cardinals and nonspecial
  {A}ronszajn trees}, Proceedings of the American Mathematical Society
  \textbf{104} (1988), no.~3, 887--897.

\bibitem{Todorcevic1981}
Stevo Todorcevic, \emph{Stationary sets, trees and continuums}, Publications de
  l'Institut Mathématique. Nouvelle Série \textbf{29(43)} (1981), 249--262.

\bibitem{todorcevic}
\bysame, \emph{Walks on ordinals and their characteristics}, Progress in
  Mathematics, vol. 263, Birkh\"auser Verlag, Basel, 2007.

\bibitem{todorcevic_torres}
Stevo Todorcevic and V{\'\i}ctor Torres~P{\'e}rez, \emph{Conjectures of {R}ado
  and {C}hang and special {A}ronszajn trees}, Mathematical Logic Quarterly
  \textbf{58} (2012), no.~4-5, 342--347.

\bibitem{viale}
Matteo Viale, \emph{The proper forcing axiom and the singular cardinal
  hypothesis}, The Journal of Symbolic Logic (2006), 473--479.

\end{thebibliography}

\end{document}